\documentclass[12 pt, a4paper]{article}
\usepackage{graphicx,mathtools,amsmath,amssymb,amsthm,enumerate, hyperref, xcolor,mathrsfs}
\usepackage{marvosym, fontawesome}
\usepackage[indent = 20 pt]{parskip}

\usepackage{url}
\usepackage{enumitem}

\newtheorem{theorem}{Theorem}[section]
\newtheorem{corollary}[theorem]{Corollary}
\newtheorem{lemma}[theorem]{Lemma}
\newtheorem{proposition}[theorem]{Proposition}
\newtheorem*{theorem*}{Theorem}

\newtheorem{letterthm}{Theorem}
\newtheorem{lettercor}[letterthm]{Corollary}

\theoremstyle{definition}
\newtheorem{definition}[theorem]{Definition}
\newtheorem*{definition*}{Definition}

\theoremstyle{remark}
\newtheorem{remark}[theorem]{Remark}
\newtheorem*{claim*}{Claim}

\newcommand{\cG}{\mathcal{G}}
\newcommand{\cH}{\mathcal{H}}

\newcommand{\cR}{\mathcal{R}}

\newcommand{\N}{\mathbb{N}}

\newcommand{\Z}{\mathbb{Z}}

\newcommand{\actson}{\curvearrowright}

\newcommand{\inv}{^{-1}}
\newcommand{\zero}{^{(0)}}

\DeclareMathOperator{\Fix}{Fix}
\DeclareMathOperator{\Iso}{Iso}

\usepackage[doi=false, url =false, isbn = false, style=alphabetic, sorting=nyt, backend=biber, maxalphanames=5, uniquelist=false]{biblatex}
\makeatletter
\def\thanks#1{\protected@xdef\@thanks{\@thanks
        \protect\footnotetext{#1}}}
\makeatother

\newlength\mylen
\newlist{case}{enumerate}{1}
\setlist[case,1]{label=\textbf{Case~\arabic*.}, 
  labelwidth=\dimexpr-\mylen-\labelsep\relax,leftmargin=0pt,align=right}

\title{The Choquet--Deny Property for Groupoids}
\author{
Tey Berendschot  \and 
Soham Chakraborty\and 
Milan Donvil\and 
Se-Jin Kim\and 
Mario Klisse \thanks{
\hspace{-2 em} \faMapMarker\hspace{0.18 em}: KU Leuven, Department of Mathematics, Leuven, Belgium \\ 
\Letter : \texttt{tey.berendschot@gmail.com, \{soham.chakraborty, milan.donvil, \\ 
\hphantom \hspace{2 em} sam.kim,~mario.klisse\}@kuleuven.be}
}
} 

\date{\today}

\begin{document}

\maketitle 

\begin{abstract}\noindent
A countable discrete group is called Choquet--Deny if for any non-degenerate probability measure on the group, the corresponding space of bounded harmonic functions is trivial. Building on the previous work of Jaworski, a complete characterization of Choquet--Deny groups was recently achieved by Frisch, Hartman, Tamuz, and Ferdowski. In this article, we extend the study of the Choquet--Deny property to the framework of discrete measured groupoids. Our primary result offers a complete characterization of this property in terms of the isotropy groups and the equivalence relation associated with the given groupoid. Additionally, we use the implications derived from our main theorem to classify the Choquet--Deny property of transformation groupoids. 
\end{abstract}

\section*{Introduction}

Motivated by analogous concepts in complex analysis, a pair $(G,\mu)$ consisting of a countable discrete group $G$ and a probability measure $\mu$ on $G$ is called \emph{Liouville} if the space of complex-valued bounded harmonic functions $f$ on $G$ that are $\mu$-harmonic, i.e., satisfying $f(g)=\sum_{h\in G}\mu(h)f(gh)$ for every $g\in G$, contains only constant functions. This property is equivalently described by the triviality of the \emph{Poisson boundary} of the pair $(G,\mu)$. Poisson boundaries were introduced by Furstenberg in a series of papers \cite{Furstenberg63a, Furstenberg63b, Furstenberg71, Furstenberg73} as a tool for studying random walks on countable discrete groups. Since then, the concept has been widely generalized and viewed from various perspectives. A significant breakthrough in the study of Poisson boundaries (and consequently the Liouville property) is the Kaimanovich--Vershik theorem in \cite{VershikKaimanovich79, KaimanovichVershik83}, which provides an affirmative answer to Furstenberg's conjecture in \cite{Furstenberg73}. Kaimanovich and Vershik demonstrated that every countable discrete amenable group admits a probability measure with a trivial associated Poisson boundary. This result was later extended to more general topological groups \cite{Rosenblatt81, SchneiderThom20}. Conversely, it is known that the Poisson boundaries of non-amenable groups are always non-trivial \cite{Furstenberg73}.

Note that the probability measure $\mu$ of a Liouville pair $(G,\mu)$ must be \emph{non-degenerate}, meaning its support generates the entire group as a semigroup. It is natural to ask for countable discrete groups for which all non-degenerate probability measures admit only constant bounded harmonic functions. Such groups are said to possess the \emph{Choquet--Deny} property. The first examples of Choquet--Deny groups were given by Blackwell in \cite{Blackwell} and by Choquet and Deny in \cite{ChoquetDeny}, who demonstrated that Abelian groups have the Choquet--Deny property. This result was later extended to virtually nilpotent groups in \cite{DynkinMaljutov61} and more generally to \emph{FC-hypercentral} groups in \cite{Jaworski}. Recall that a group is called FC-hypercentral if each of its quotients contains a non-trivial finite conjugacy class. Despite the discovery of many examples of amenable groups that do not possess the Choquet--Deny property \cite{Kaimanovich83b, KaimanovichVershik83, Erschler04a, Erschler04b}, a complete characterization of this property remained an open question for a long time. Only recently, Frisch, Hartman, Tamuz, and Ferdowski provided a significant breakthrough by proving that every Choquet--Deny group is FC-hypercentral \cite[Theorem 1]{Frisch2019}. The analogous concept for tracial von Neumann algebras with separable preduals was first investigated by Das and Peterson in \cite{DasPeterson22}. Later, Zhou provided a complete characterization of Choquet--Deny von Neumann algebras in terms of their types in \cite{Zhou24}.

In this article, we initiate the study of the Choquet--Deny property in the context of measured groupoids. A \emph{groupoid} is a small category where all morphisms are invertible, thus generalizing the notion of groups. The concept is motivated by its profound connections to various fields such as dynamics and ergodic theory \cite{Xin18}, operator algebras \cite{Renault, RenaultCartan, Anantharaman}, non-commutative geometry \cite{Connes94}, and homotopy type theory \cite{hottbook}. In particular the connection between measured groupoids and von Neumann algebras have been studied extensively, for example in \cite{Mackey49, Feldman-Moore-1, Feldman-Moore-2, Hahn78, SutherlandTakesaki85, Yamanouchi94}. In \cite{Kaimanovich05}, Kaimanovich introduced the \emph{fiberwise Liouville} property for measured groupoids. Drawing an analogy to invariant Markov operators on groups corresponding to random walks, he considered \emph{invariant Markov operators} on groupoids that act fiberwise with respect to a Haar system. A measured groupoid equipped with such an invariant Markov operator is called fiberwise Liouville if almost all of its fiberwise actions admit no non-trivial bounded harmonic functions (for precise definitions see Subsection \ref{MeasuredGroupoids} and Section \ref{MainSection}). Similar to the group case, this notion is closely related to amenability: results in \cite{Kaimanovich05, ChuXin2018} (see also \cite{Buehler}) imply that a measured groupoid is \emph{amenable} in the sense of \cite{Renault, Anantharaman} if and only if it admits an invariant Liouville Markov operator.

In the context of measured groupoids, the analogue of countable discrete groups are \emph{discrete measured groupoids}. Building on the previous discussion, we introduce the following definition.

\begin{definition*} 
A ($\sigma$-finite) discrete measured groupoid is called \emph{Choquet--Deny} if it is fiberwise Liouville for every non-degenerate invariant Markov operator. 
\end{definition*}

In this article, we assume all measures to be $\sigma$-finite. Under this assumption, we provide a complete characterization of the Choquet--Deny property in terms of the groupoid's isotropy groups and the associated equivalence relation.

\begin{letterthm}\label{Main Theorem CD} 
A discrete measured groupoid $(\mathcal{G},\mu)$ is Choquet--Deny if and only if the following two conditions hold: 
\begin{enumerate}[label={(\arabic*)}] 
\item The countable Borel equivalence relation associated with $(\mathcal{G},\mu)$ has finite orbits $\mu$-almost everywhere. 
\item Almost all isotropy groups of $\mathcal{G}$ are Choquet--Deny. 
\end{enumerate}
\end{letterthm}

Our proof of Theorem \ref{Main Theorem CD} is divided into several steps and extensively utilizes the construction presented in \cite{Frisch2019}. The theorem notably implies that a discrete measured groupoid is Choquet--Deny if and only if it admits no icc quotients apart from finite equivalence relations, see Proposition \ref{Proposition:CD and icc quotients}. The concept of icc discrete measured groupoids was introduced by the first four authors in \cite{berendschot2024factoriality}.

Measured groupoids admit an analogue of the Poisson boundary, studied by Kaimanovich in \cite{Kaimanovich05}. Similar to groups, the triviality of a measured groupoid's Poisson boundary can be understood in terms of the fiberwise Liouville property. Theorem~\ref{Main Theorem CD} illustrates that discrete measured groupoids with the Choquet--Deny property are relatively rare, as demonstrated by the following result:
\begin{lettercor}\label{Main Corollary}
Let $\Gamma$ be a countable group acting on a Borel probability space $(X,\mu)$. If $\Gamma$ is not Choquet--Deny or if $\Gamma \actson (X,\mu)$ does not have finite orbits $\mu$-almost everywhere then the transformation groupoid $(\Gamma\ltimes X,\mu)$ admits a non-trivial Poisson boundary. 
\end{lettercor}
This is an indication that the corresponding Poisson boundaries present in many cases interesting objects of study.

\subsection*{Structure} 
The article is structured as follows. In Section \ref{preliminaries} we remind the reader of basic notions surrounding the theory of measured groupoids, Borel equivalence relations, semi-direct product groupoids, Markov chains, and martingales. In Section \ref{MainSection} we introduce and explore the Choquet--Deny property for groupoids: first we characterize this property within the realm of countable Borel equivalence relations and demonstrate that a non-trivial icc measured field of groups does not exhibit the Choquet--Deny property. In combination with results on the preservation of the Choquet--Deny property under suitable quotients, this allows us to obtain the ``only if'' direction of Theorem \ref{Main Theorem CD} in Subsection \ref{ChoquetDenyUnderQuotients}. The converse direction is then attained in Subsection \ref{GroupoidMarkovChains} by employing Doob's optional stopping theorem on Markov chains derived from discrete measured groupoids. In the final Subsection \ref{examples} we collect implications of Theorem \ref{Main Theorem CD} by considering additional examples and the proof of Theorem~\ref{theorem: transformation groupoids}, which immediately implies Corollary~\ref{Main Corollary}.

\subsection*{Acknowledgements}
The authors are supported by the Research Foundation Flanders (FWO): T.~Berendschot is supported by PhD grant 1101324N, S.~Chakraborty is supported by research project G090420N, M.~Donvil is supported by PhD grant 1162024N, S.~Kim is supported by research project G085020N, and M.~Klisse is supported by postdoctoral grant 1203924N. Additionally, S.~Kim is supported by the internal KU Leuven funds project number C14/19/088.

\section{Preliminaries} \label{preliminaries}

\subsection{Measured groupoids} \label{MeasuredGroupoids}

Given a groupoid $\mathcal{G}$ we denote its \emph{unit space} by $\mathcal{G}^{(0)}$ and write $s$ and $t$ for the \emph{source} and \emph{target maps} respectively. The \emph{source} and \emph{target fibers} of an element $x\in\mathcal{G}^{(0)}$ are given by $\mathcal{G}_{x}:=s^{-1}(\{x\})$ and $\mathcal{G}^{x}:=r^{-1}(\{x\})$. The \emph{isotropy subgroupoid} of $\mathcal{G}$ is defined as $\text{Iso}(\cG):=\{g \in \cG \mid s(g) = t(g)\}$ and $\mathcal{G}^{(2)}:=\{(g,h)\in\mathcal{G}\times\mathcal{G} \mid  s(g)=t(h)\}$ denotes the subset of \emph{composable pairs}.

A \emph{discrete Borel groupoid} is a groupoid $\mathcal{G}$ that is also a standard Borel space for which $\mathcal{G}^{(0)}\subseteq\mathcal{G}$ is a Borel subset, for which the structure maps (i.e., the source and target maps, the multiplication and the inverse) are Borel measurable functions, and for which $s$ and $t$ are \emph{countable-to-one} in the sense that all source and target fibers are at most countably infinite. Given such a groupoid and a Borel measure $\mu$ on $\mathcal{G}^{(0)}$, we introduce two measures $\mu_{s}$ and $\mu_{t}$ on $\mathcal{G}$ via
\[
\mu_{s}(A):=\int_{\mathcal{G}^{(0)}}\#(s^{-1}(x)\cap A)d\mu(x)
\]
and
\[
\mu_{t}(A):=\int_{\mathcal{G}^{(0)}}\#(t^{-1}(x)\cap A)d\mu(x)
\]
for a Borel subset $A\subseteq\mathcal{G}$. The pair $(\mathcal{G},\mu)$ is a called a \emph{discrete measured groupoid} if $(\mathcal{G}^{(0)},\mu)$ is a standard probability space and if the measures $\mu_{s}$ and $\mu_{t}$ are \emph{equivalent} (i.e., if their null sets coincide). In this article we will always assume the groupoid $\mathcal{G}$ to be discrete and the measure $\mu$ to be $\sigma$-finite. Since every $\sigma$-finite measure is equivalent to a probability measure, we may furthermore restrict to probability measures.

Let $(\pi_{g})_{g\in\mathcal{G}}\subseteq\text{Prob}(\mathcal{G})$ be a family of Borel probability measures on $\mathcal{G}$. We say that the family is \emph{Borel} if for any non-negative Borel function $f:\mathcal{G}\rightarrow\mathbb{R}$ the map $g\mapsto\int_{\mathcal{G}}f(h)\pi_{g}(h)$ is again Borel. Associated to a Borel family of probability measures $(\pi_{g})_{g\in\mathcal{G}}$ is its \emph{Markov operator} $P$ on $L^{\infty}(\mathcal{G},\mu_{s})$, defined by $Pf(g)=\int_{\mathcal{G}}f(h)\pi_{g}(h)$ for $f\in L^{\infty}(\mathcal{G},\mu_{s})$, $g\in\mathcal{G}$. We say that the family $(\pi_{g})_{g\in\mathcal{G}}$ (or equivalently, the Markov operator $P$) is \emph{invariant} if $g\cdot\pi_{h}=\pi_{gh}$ for any pair $(g,h)\in\mathcal{G}^{(2)}$, where $(g\cdot\pi_{h})(A):=\pi(g^{-1}A)$. Note that in this case the measure $\pi_{g}$ must be concentrated on $\mathcal{G}^{t(g)}$. Moreover, the family $(\pi_{g})_{g\in\mathcal{G}}$ is uniquely determined by the $\pi_{x}$, $x\in\mathcal{G}^{(0)}$ via $\pi_{g}=g\cdot\pi_{s(g)}$ for $g\in\mathcal{G}$ \cite[Proposition 3.4]{Kaimanovich05}. We may hence view an invariant Markov operator $P$ as a collection of Markov operators $(P_{x})_{x\in\mathcal{G}^{(0)}}$ associated to probability measures $(d_{x}P)_{x\in\mathcal{G}^{(0)}}$ with $d_{x}P\in\text{Prob}(\mathcal{G}^{x})$ for $x\in\mathcal{G}^{(0)}$ and where $P_{x}\in\mathcal{B}(\ell^{\infty}(\mathcal{G}^{x}))$ is defined by 
\[
P_{x}f(g):=\sum_{h\in\mathcal{G}^{s(g)}}\left(d_{s(g)}P(h)\right)f(gh)
\]
for $f\in\ell^{\infty}(\mathcal{G}^{x})$, $g\in\mathcal{G}^{x}$. We recursively obtain Markov operators $P^{n}$ for every $n\in\mathbb{N}$ by setting $P^{n}f:=P(P^{n-1}f)$. The associated probability measures are given by 
\[
d_{x}P^{n}(g)=\sum_{h\in\mathcal{G}^{x}}\left(d_{x}P(h)\right)\left(d_{s(h)}P^{n-1}(h^{-1}g)\right)
\]
for $x\in\mathcal{G}^{(0)}$, $g\in\mathcal{G}^{x}$.

Let  $(\mathcal{G},\mu)$ be a discrete measured groupoid equipped with an invariant Markov operator $P$ and let $x\in\mathcal{G}^{(0)}$. The subspace $H^{\infty}(\mathcal{G}^{x},P_{x})\subseteq\ell^{\infty}(\mathcal{G}^{x})$ of \emph{$P_{x}$-harmonic functions} is given by 
\[
H^{\infty}(\mathcal{G}^{x},P_{x}):=\{f\in\ell^{\infty}(\mathcal{G}^{x}) \mid  P_{x}f=f\} \,.
\]
We say that $P_{x}$ is \emph{Liouville}, if the space $H^{\infty}(\mathcal{G}^{x},P_{x})$ is trivial in the sense that it only consists of constant functions.

Following \cite{Kaimanovich05}, we introduce the Liouville property for groupoids.
\begin{definition}[{\cite[Definition 4.1]{Kaimanovich05}}] 
An invariant Markov operator $P$ on a discrete measured groupoid $(\mathcal{G},\mu)$ is called \emph{fiberwise Liouville} if for $\mu$-almost every $x\in\mathcal{G}^{(0)}$ the operator $P_{x}\in\mathcal{B}(\ell^{\infty}(\mathcal{G}^{x}))$ is Liouville.\\
The groupoid is called \emph{Liouville}, if it admits a fiberwise Liouville invariant Markov operator. \end{definition}

For countable groups (and more generally Hausdorff second-countable topological groups \cite{SchneiderThom20}) it is well-known that amenability is equivalent to the existence of a Liouville Borel probability measure on the group, see \cite{KaimanovichVershik83}. Similarly, Kaimanovich's results in \cite{Kaimanovich05} (see also \cite{Buehler}) imply in combination with the ones by Chu and Li in \cite{ChuXin2018} that a (discrete) measured groupoid is amenable in the sense of \cite{Renault, Anantharaman} if and only if it admits an invariant Liouville Markov operator.

\subsection{Countable Borel equivalence relations}

\label{Sec: countable Borel equivalence relations}

In this section we state preliminaries on Borel equivalence relations. For detailed proofs and a rigorous treatment, we refer the reader to \cite{Feldman-Moore-1, Feldman-Moore-2}. 

A discrete measured groupoid $(\mathcal{G},\mu)$ with unit space $X:=\mathcal{G}^{(0)}$ and the property that $\Iso(\cG)$ is trivial is called a \textit{non-singular countable Borel equivalence relation}. We usually denote such an equivalence relation on the standard Borel space $(X,\mu)$ by $\mathcal{R}$. For a Borel subset $E\subseteq X$, we write $\mathcal{R}(E):=\{y\in X \mid (x,y)\in\mathcal{R}\text{ for some }x\in E\}$ for its \emph{$\mathcal{R}$-saturation}. An equivalence relation $\mathcal{R}$ on $(X,\mu)$ is called \textit{ergodic} if every Borel subset $E\subseteq X$ satisfying $\mathcal{R}(E)=E$ is either a $\mu$-null set or a $\mu$-co-null set.

An equivalence relation $\mathcal{R}$ is \textit{finite} if its orbits are finite. It can be checked that any finite ergodic equivalence relation is isomorphic to the full equivalence relation on a finite set. Note that the \emph{restriction} $\mathcal{R}|_{X_{1}}$ of $\mathcal{R}$ to a Borel subset $X_{1}\subseteq X$ is again an equivalence relation on $(X_{1},\mu|_{X_{1}})$. Such an equivalence relation will be called a \emph{summand} of $\mathcal{R}$. A countable Borel equivalence relation $\mathcal{R}$ on $(X,\mu)$ is said to be of \textit{type I$_{n}$} for $n\in\mathbb{N}\cup\{\infty\}$ if it is isomorphic to an equivalence relation $\mathcal{R}^{\prime}$ on $S\times Y$ where $S$ is a set of $n$ elements equipped with the full $\sigma$-algebra and the counting measure, $Y$ is a standard Borel space and where $(s,y)\sim_{\mathcal{R}^{\prime}}(s^{\prime},y^{\prime})$ if and only if $y=y^{\prime}$. The equivalence relation $\mathcal{R}$ is said to be of \textit{type I} or \textit{discrete} if $X$ admits a Borel partition $X=\bigsqcup_{n\in\mathbb{N}\cup\{\infty\}}X_{n}$ such that each $\mathcal{R}|_{X_{n}}$ is of type I$_{n}$. 

Recall that an equivalence relation $\cR$ on $(X,\mu)$ comes equipped with the measures $\mu_{s}$ and $\mu_{t}$. We say that $\cR$ is \textit{invariant} with respect to $\mu$ if $\mu_{s} = \mu_{t}$, and that it is \textit{quasi-invariant} with respect to $\mu$ if $\mu_{s} \sim \mu_{t}$. If $\cR$ on $(X,\mu)$ is invariant with respect to an equivalent probability measure $\nu\sim\mu$, it is called \emph{probability measure preserving} (\emph{p.m.p.}). For a non-type I ergodic equivalence relation $\mathcal{R}$, we have the following classification: it is said to be of \emph{type II$_{1}$} if it is p.m.p., it is said to be of \emph{type II$_{\infty}$} if it is not of type II$_{1}$ and invariant with respect to an equivalent infinite measure, and it is said to be of \emph{type III} if $\mathcal{R}$ is not invariant with respect to any $\mu$-equivalent measure.

By \cite[Proposition 3.1]{Feldman-Moore-1}, for any countable Borel equivalence relation $\mathcal{R}$ on $(X,\mu)$, the measure space $X$ admits a Borel partition
\[
X=\left(\bigsqcup_{n\in\mathbb{N}}X_{n}\right)\sqcup X_{\infty}\sqcup X^{C}
\]
into $\mathcal{R}$-invariant Borel subsets, where $\mathcal{R}|_{X_{n}}$ is of type $\text{I}_{n}$, $\mathcal{R}|_{X_{\infty}}$ is of type $\text{I}_{\infty}$, and where $\mathcal{R}|_{X^{C}}$ is not of type I. We will denote the Borel subset $\bigsqcup_{n\in\mathbb{N}}X_{n}$ by $X_{\text{fin}}$ and call $\mathcal{R}|_{X_{\text{fin}}}$ the \textit{finite part} of $\mathcal{R}$.

Recall that an equivalence relation $\mathcal{R}$ on a standard probability space $(X,\mu)$ is called \textit{smooth} if it admits a \textit{Borel transversal}, i.e., there exists a Borel subset $E\subseteq X$ which intersects every $\mathcal{R}$-orbit exactly once. The following result is well known and for a  quick proof of the various implications we refer the reader to \cite[Chapter 2]{CalderoniLectureNotes09}.

\begin{theorem}\label{Theorem: type I eq rel} Let $\mathcal{R}$ be a countable Borel equivalence relation on a standard probability space $(X,\mu)$. The following statements are equivalent: 
\begin{enumerate}[label=(\roman*)]
\item $\mathcal{R}$ is of type I.
\item $\mathcal{R}$ is smooth. 
\item $\mathcal{R}$ admits a \emph{Borel selector}, i.e., a Borel function $f:X\rightarrow X$ such that $(f(x),x)\in\mathcal{R}$ for all $x\in X$ and $f(x)=f(y)$ if $(x,y)\in\mathcal{R}$.
\item The quotient $X/\mathcal{R}$ is a standard Borel space. 
\end{enumerate}
\end{theorem}

Ergodic equivalence relations can be thought of as building blocks of general Borel equivalence relations. We make this more precise in the next theorem which was proved by Dang Ngoc Nghiem in \cite{Nghiem75}. Our formulation is similar to the one in \cite[Proposition 3.2]{Feldman-Moore-1} except that we do not assume the field of standard Borel spaces to be constant. Instead we relate it to the unified new framework of measurable fields of separable structures introduced in \cite[Appendix A]{Wouters23}, \cite{VaesWouters24}. We note here that there occurs a more general ergodic decomposition theorem for measured groupoids with a Haar system in \cite[Theorem 6.1]{Hahn78} which in the case of principal discrete measured groupoids corresponds to our setting. As noted by Hahn in \cite{Hahn78}, the theorem was also proved independently by Ramsay in \cite{Ramsay80}. Note that the space of essentially bounded measurable functions $L^\infty(X, \mu)$ on a standard probability space $(X, \mu)$ forms an Abelian von Neumann algebra. If $\cR$ is an equivalence relation on $(X, \mu)$, we denote by $L^{\infty}(X,\mu)^{\mathcal{R}}$ its subalgebra of $\cR$-invariant functions.

Now suppose that $(Z,\eta)$ is a standard probability space, let $Y:=(Y_{z})_{z\in Z}$ be a measurable field of standard Borel spaces as in \cite[Definition A.2.1]{Wouters23} and let $\pi:Y\rightarrow Z$ be the Borel projection map. Suppose furthermore that $(\nu_{z})_{z\in Z}$ is a measurable field of probability measures on $Y$. Then a family of ergodic equivalence relations $\mathcal{R}:=(\mathcal{R}_{z})_{z\in Z}$ that are quasi-invariant with respect to $\nu_{z}$ forms a \textit{measurable field} of Borel equivalence relations if $\mathcal{R}$ is a Borel subset of $Y\times Y$. By abuse of notation we denote by $Y$ and $\mathcal{R}$ the measurable fields as well as the standard Borel structures on the respective disjoint unions. Let $\nu$ be the probability measure on $Y$ obtained by integration of $(\nu_{z})_{z\in Z}$ with respect to $\eta$. Then we can define a countable Borel equivalence relation $\mathcal{R}$ on $(Y,\nu)$ via $(y,y^{\prime})\in\mathcal{R}$ if and only if $\pi(y)=\pi(y^{\prime})$ and $(y,y')\in\mathcal{R}_{\pi(y)}$. We call $\mathcal{R}$ the \textit{direct integral of the field $(\mathcal{R}_{z})_{z\in Z}$} and denote it by $\int_{Z}^{\oplus}\mathcal{R}_{z}d\eta(z)$.

\begin{theorem}[{\cite[Theorem 6.1]{Hahn78}, \cite[Proposition 3.2]{Feldman-Moore-1}}] \label{Thm: ergodic decomposition theorem} \label{Thm: equivalence relations admit an ergodic decomposition} Let $\mathcal{R}$ be a countable Borel equivalence relation on a standard probability space $(X,\mu)$ and let $(Z,\eta)$ be a standard probability space such that the Abelian von Neumann algebra $L^{\infty}(X,\mu)^{\mathcal{R}}$ is isomorphic to $L^{\infty}(Z,\eta)$.\\
Then $\mathcal{R}$ is isomorphic to $\int_{Z}^{\oplus}\mathcal{R}_{z}d\eta(z)$ for a measurable field of countable non-singular equivalence relations $(\mathcal{R}_{z})_{z\in Z}$ on a Borel field of standard Borel spaces and probability measures $Y:=(Y_{z},\nu_{z})_{z\in Z}$ such that, denoting the projection $Y\rightarrow Z$ by $\pi$, any invariant Borel subset of $X$ is of the form $\pi^{-1}(A)$ for a Borel set $A\subseteq Z$ up to measure zero and such that for $\eta$-almost every $z\in Z$ the equivalence relation $\mathcal{R}_{z}$ is ergodic.\\
Moreover, if $X_{k}:=\pi^{-1}(Z_{k})$ are the summands of $\mathcal{R}$ corresponding to the different types, then $\mathcal{R}_{z}$ is of type $k$ for $\eta$-almost every $z\in Z_{k}$. \end{theorem}

An equivalence relation $\mathcal{R}$ on $(X,\mu)$ is called \emph{hyperfinite} if it can be written as a countable union $\bigcup_{n}\mathcal{R}_{n}$ of finite equivalence relations. As a consequence of the main result in \cite{Connes_Feldman_Weiss_1981}, an equivalence relation is hyperfinite if and only if it is amenable. Notice that if $\mathcal{R}$ admits an ergodic decomposition $\int_{Z}^{\oplus}\mathcal{R}_{z}d\eta(z)$, then $\mathcal{R}$ is hyperfinite if and only if $\mathcal{R}_{z}$ is hyperfinite for $\eta$-almost every $z\in Z$. It is furthermore easy to see that every type I equivalence relation (in particular every finite equivalence relation) is hyperfinite.

Up to isomorphism there is for every $n\in\mathbb{N}\cup\{\infty\}$ a unique ergodic equivalence relation of type I$_{n}$ given by considering the full equivalence relation on a set of $n$ elements with respect to the counting measure. There also exists a unique hyperfinite ergodic equivalence relation of type II$_{1}$ and of type II$_{\infty}$.

For an ergodic equivalence relation $\mathcal{R}$ on $(X,\mu)$, recall the notion of its \emph{Maharam extension}. This is an equivalence relation $c(\mathcal{R})$ on $X\times\mathbb{R}$ together with a measure scaling $\mathbb{R}$-action. The restriction of the induced $\mathbb{R}$-action to the Abelian von Neumann algebra $L^{\infty}(X\times\mathbb{R})^{c(\mathcal{R})}$ is called the \textit{associated Krieger flow} of $\mathcal{R}$. For type I and type II ergodic equivalence relations, the flow is always isomorphic to the translation action $\mathbb{R}\curvearrowright\mathbb{R}$. Ergodic hyperfinite type III equivalence relations are completely classified by their associated flows and every ergodic flow arises as the associated flow of some hyperfinite ergodic type III equivalence relation. For a detailed treatment of equivalence relations, Maharam extensions and flows we refer the reader to \cite[Chapter 1.6]{Bramthesis}.

Finally, recall from \cite{JonesSchmidt87} that a countable Borel equivalence relation $\mathcal{R}$ is called \textit{stable} if $\mathcal{R}\times\mathcal{R}_{0}$ is isomorphic to $\mathcal{R}$ for the unique p.m.p hyperfinite ergodic equivalence relation $\mathcal{R}_{0}$. It is well known (and follows in particular from \cite[Theorem 3.4]{JonesSchmidt87}) that every hyperfinite ergodic equivalence relation that is not of type I is stable.

\subsection{Semi-direct product groupoids}

A discrete measured groupoid $(\mathcal{G},\mu)$ comes with an associated countable Borel equivalence relation $\mathcal{R}$ on $(X,\mu)$ with $X:=\mathcal{G}^{(0)}$, defined by 
\[
\mathcal{R}\coloneqq\{(t(g),s(g)) \mid g\in\mathcal{G}\}\,.
\]
Given a countable Borel equivalence relation $\mathcal{R}$ on a standard probability space $(X,\mu)$, one way to construct a discrete measured groupoid with associated equivalence relation $\mathcal{R}$ (that is not simply $\cR$ itself) is by looking at actions of $\mathcal{R}$ on bundles of groups and considering the corresponding semi-direct products. We now make this more precise. Suppose that $\Gamma:=(\Gamma_{x})_{x\in X}$ is a measurable field of discrete countable groups. An action of $\mathcal{R}$ on $\Gamma$ is given by a family of group isomorphisms $\delta_{(y,x)}:\Gamma_{x}\rightarrow\Gamma_{y}$ for all $(y,x)\in\mathcal{R}$ such that outside a $\mu$-null set the equalities $\delta_{(z,y)}\circ\delta_{(y,x)}=\delta_{(z,x)}$, $\delta_{(y,x)}^{-1}=\delta_{(x,y)}$ and $\delta{(x,x)}=\text{id}_{\Gamma_{x}}$ hold. 

The \emph{semi-direct product groupoid} $\mathcal{G}:=\Gamma\rtimes_{\delta}\mathcal{R}$ consists of elements of the form $(g,(y,x))$ for all $g\in\Gamma_{x}$ and $(y,x)\in\mathcal{R}$. Heuristically, these elements can be viewed as ``arrows'' from $x$ to $y$. Accordingly we define the source and targets maps on $\mathcal{G}$ via $s(g,(y,x)):=x$ and $t(g,(y,x)):=y$. The unit space consists of the elements $\{(e,x,x) \; | \; x \in X\}$ and we identify it with $X$. Multiplication and the inverse in $\mathcal{G}$ are respectively defined by
\[
(h,(z,y))\circ(g,(y,x))=(\delta_{(x,y)}(h)\cdot g,(z,x))
\]
and
\[
(g,(y,x))^{-1}=(\delta_{(y,x)}(g^{-1}),(x,y))\,.
\]

Recall that a countable Borel equivalence relation $\mathcal{R}$ on $(X,\mu)$ is called \textit{treeable} if there exists a Borel graph $T$ on $(X,\mu)$ such that $T\subseteq\mathcal{R}$ and if for $\mu$-almost every $x\in X$ the connected component of $x$ in $T$ is a tree with vertex set $T\cdot x$. Following \cite[Proposition 6.5]{Popa-Shlyakhtenko-Vaes20}, an equivalence relation $\mathcal{R}$ on $(X,\mu)$ is treeable if and only if any discrete measured groupoid $(\mathcal{G},\mu)$ with unit space $\mathcal{G}^{(0)}=X$ and associated equivalence relation $\mathcal{R}$ is isomorphic to a semi-direct product as above. One immediate consequence of \cite{Connes_Feldman_Weiss_1981} is that any amenable equivalence relation is treeable and hence by \cite[Proposition 6.5]{Popa-Shlyakhtenko-Vaes20}, we obtain the following statement.

\begin{proposition}\label{Proposition: semi-direct product groupoid} Let $(\mathcal{G},\mu)$ be a discrete measured groupoid. Let $\Gamma:=(\Gamma_{x})_{x\in\mathcal{G}^{(0)}}$ be the corresponding isotropy bundle and let $\mathcal{R}$ be the associated equivalence relation. If $\cR$ is amenable, then $\mathcal{G}$ is isomorphic to $\Gamma\rtimes_{\delta}\mathcal{R}$ for an action $\delta$ of $\mathcal{R}$ on the isotropy bundle $\Gamma$. \end{proposition}

\subsection{Markov chains and martingales}

The final step in the proof of Theorem \ref{Main Theorem CD} requires the study of Markov chains on discrete measured groupoids, as well as the theory of martingales. For the convenience of the reader, we introduce the basic notions of these topics.

\subsubsection{Markov chains} \label{MarkovChains}

Consider a stochastic matrix $P:=(P(x,y))_{x,y\in X}$ on a countable set $X$, i.e., $P(x,y)\geq0$ and $\sum_{z\in X}P(x,z)=1$ for all $x,y\in X$. Similar to the setting in Subsection \ref{MeasuredGroupoids}, such a matrix induces a \emph{Markov operator} on $\ell^{\infty}(X)$ by mapping a function $f\in\ell^{\infty}(G)$ to $x\mapsto\sum_{y\in X}p(x,y)f(y)$. By abuse of notation we will denote the Markov operator corresponding to the matrix $P$ by $P$ as well. A \emph{Markov chain} with \emph{transition matrix} $P$ and \emph{starting point} $x\in X$ is a sequence of random variables $(X_{i})_{i\in\mathbb{N}}$ on a probability space $(\Omega,\mathcal{F},\mu_{x})$ with values in $X$ such that 
\[
\mu_{x}\left(\{\omega\in\Omega \mid X_{0}(\omega)=x_{0},\ldots,X_{i}(\omega)=x_{i}\}\right)=\delta_{x,x_{0}}\prod_{j=0}^{i-1}P(x_{j},x_{j+1})
\]
for any finite sequence of elements $x_{0},\ldots,x_{n}\in X$. Here $\delta$ denotes the \emph{Kronecker delta function}. The Markov chain $(X_{i})_{i\in\mathbb{N}}$ is said to be \emph{irreducible}, if for any two elements $x,y\in X$ there exists a sequence $x_{0}:=x$, $x_{1},\ldots,x_{i-1},$ $x_{i}:=y$ in $X$ with $\mu_{x}\left(\left\{ \omega\in\Omega\mid X_{0}(\omega)=x_{0},\ldots,X_{i}(\omega)=x_{i}\right\} \right)>0$. Equivalently, $(X_{i})_{i\in\mathbb{N}}$ is irreducible if the matrix $P$ is \emph{irreducible} in the sense that for any two elements $x,y\in X$ there exists a natural number $i$ with $P^{i}(x,y)>0$.

For more details on Markov chains see \cite[Chapter 6]{Lalley23} or also \cite[Chapter 3]{Yadin23}.

\subsubsection{Martingales}

Let $(\Omega,\mathcal{F},\mu)$ be a probability space and $\mathcal{F}^{\prime}\subseteq\mathcal{F}$ a sub-$\sigma$-algebra. For any $\mathcal{F}$-measurable, integrable random variable $X:(\Omega,\mathcal{F},\mu)\rightarrow\mathbb{R}$ there exists an $\mu$-almost surely unique $\mathcal{F}^{\prime}$-measurable, integrable random variable $Y:(\Omega,\mathcal{F}^{\prime},\mu)\rightarrow\mathbb{R}$ such that 
\[
\mathbb{E}_{\mu}[X1_{F}]=\int_{F}X(\omega)d\mu(\omega)=\int_{F}Y(\omega)d\mu(\omega)=\mathbb{E}_{\mu}[Y1_{F}]
\]
for every $F\in\mathcal{F}^{\prime}$ \cite[Appendix A9]{Lalley23}. The random variable $Y$ is usually called the \emph{conditional expectation} of $X$ with resepct to $\mathcal{F}^{\prime}$ and we denote it by $\mathbb{E}_{\mu}[X\mid\mathcal{F}^{\prime}]$. The map $X\mapsto\mathbb{E}_{\mu}[X\mid\mathcal{F}^{\prime}]$ is linear and restricts to the identity on $\mathcal{F}^{\prime}$-measurable random variables. Note that $\mathbb{E}_{\mu}[\mathbb{E}_{\mu}[X\mid\mathcal{F}^{\prime}]]=\mathbb{E}_{\mu}[X]$ for every $\mathcal{F}$-measurable, integrable random variable $X$.

A \emph{filtration} $(\mathcal{F}_{i})_{i\in\mathbb{N}}$ is a nested sequence of sub-$\sigma$-algebras $\mathcal{F}_{0}\subseteq\mathcal{F}_{1}\subseteq\ldots\subseteq\mathcal{F}$. A sequence $(X_{i})_{i\in\mathbb{N}}$ of real-valued random variables on $(\Omega,\mathcal{F},\mu)$ is called \emph{adapted} to $(\mathcal{F}_{i})_{i\in\mathbb{N}}$ if for every $i\in\mathbb{N}$ the random variable $X_{i}$ is $\mathcal{F}_{i}$-measurable.

\begin{definition} Let $(\Omega,\mathcal{F},\mu)$ be a probability space and $(\mathcal{F}_{i})_{i\in\mathbb{N}}$ a filtration. An adapted sequence $(X_{i})_{\in\mathbb{N}}$ of real-valued random variables is called a \emph{martingale}, if $\mathbb{E}_{\mu}[|X_{i}|]<\infty$ and $\mathbb{E}_{\mu}[X_{i+1}\mid\mathcal{F}_{i}]=X_{i}$ $\mu$-almost surely for all $i\in\mathbb{N}$. \end{definition}

We will make use of \emph{Doob's optional stopping formula}. The version we will use occurs in \cite[Theorem 2.3.3]{Yadin23}.

\begin{definition} Let $(\Omega,\mathcal{F},\mu)$ be a probability space and $(\mathcal{F}_{i})_{i\in\mathbb{N}}$ a filtration. A random variable $T:(\Omega,\mathcal{F},\mu)\rightarrow\mathbb{Z}_{+}\cup\{\infty\}$ is called a \emph{stopping time} for $(\mathcal{F}_{i})_{i\in\mathbb{N}}$, if $T^{-1}(\{i\})\in\mathcal{F}_{i}$ for every $i\in\mathbb{N}$. If $(X_{i})_{i\in\mathbb{N}}$ is a martingale, we denote the random variable $(\Omega,\mathcal{F},\mu)\rightarrow\mathbb{R},\omega\mapsto X_{T(\omega)}(\omega)$ by $X_{T}$. \end{definition}

\begin{theorem}[Doob's optional stopping formula] \label{DoobTheorem} Let $(\Omega,\mathcal{F},\mu)$ be a probability space, $(\mathcal{F}_{i})_{i\in\mathbb{N}}$ a filtration, $(X_{i})_{i\in\mathbb{N}}$ a matingale and $T$ a stopping time for $(\mathcal{F}_{i})_{i\in\mathbb{N}}$. Assume that $T <\infty$ $\mu$-almost surely, that $(X_{i})_{i\in\mathbb{N}}$ is \emph{uniformly integrable} in the sense that $\lim_{K\rightarrow\infty}\sup_{i\in\mathbb{N}}\mathbb{E}_{\mu}[|X_{i}|\cdot1_{\{\omega\in\Omega\mid|X_{i}(\omega)|>K\}}]=0$ and that $\mathbb{E}_{\mu}[|X_{T}|]<\infty$. Then, 
\[
\mathbb{E}_{\mu}[X_{0}]=\int_{\Omega}X_{0}(\omega)d\mu(\omega)=\int_{\Omega}X_{T}(\omega)d\mu(\omega)=\mathbb{E}_{\mu}[X_{T}]\,.
\]
\end{theorem}

More information on martingales can be found in \cite[Chapter 8]{Lalley23} or \cite[Chapter 2]{Yadin23}.

\section{The Choquet--Deny property for groupoids} \label{MainSection}

Let $\Gamma$ be a countable discrete group, equipped with a probability measure $\mu\in\text{Prob}(\Gamma)$. Analogous to the general definition in Subsection \ref{MeasuredGroupoids}, a function $f$ on $\Gamma$ is called \emph{$\mu$-harmonic} if $f(g)=\sum_{h\in G}\mu(h)f(gh)$ for all $g\in\Gamma$. As before, the pair $(\Gamma,\mu)$ is called \emph{Liouville} if the space of all bounded $\mu$-harmonic functions is trivial.

A probability measure on a countable discrete group is \emph{non-degenerate}, if its support generates the whole group as a semigroup. Initiated by Choquet and Deny in \cite{ChoquetDeny}, the characterization of groups that are \emph{Choquet--Deny} in the sense that all non-degenerate probability measures on them are Liouville was a long-standing open question. Building on Jaworski's results in \cite{Jaworski}, a full characterization was finally obtained by Frisch, Hartman, Tamuz and Ferdowski in \cite{Frisch2019}. The main result in \cite{Frisch2019} states that a countable discrete group is Choquet--Deny if and only if it is FC-hypercentral. Recall that a group is called \emph{FC-hypercentral} if it admits no \emph{icc} quotients, i.e., no quotients whose conjugacy classes are all infinite.

Motivated by the notions in the group setting, we introduce the following definitions.

\begin{definition} Let $(\mathcal{G},\mu)$ be a discrete measured groupoid and $P$ a Markov operator on $L^\infty(\mathcal{G},\mu_s)$. For $x\in\mathcal{G}^{(0)}$ the \emph{support} of $P_{x}$ is given by the set of all elements $g\in\mathcal{G}^{x}$ with $d_{x}P(g)>0$ and we denote it by $\text{supp}(P_{x})$.\\
The Markov operator $P$ is called non-degenerate, if $\bigcup_{n\in\mathbb{N}}\text{supp}(P_{x}^{n})=\mathcal{G}^{x}$ for all $x\in\mathcal{G}^{(0)}$. \end{definition}

\begin{definition} \label{ChoquetDeny} A discrete measured groupoid $(\mathcal{G},\mu)$ is called \emph{Choquet--Deny} if it is fiberwise Liouville for every non-degenerate invariant Markov operator $P$. \end{definition}

Note that a countable discrete group is Choquet--Deny as a group if and only if it is Choquet--Deny in the sense of Definition \ref{ChoquetDeny}.

In the proofs of the later sections it will be convenient to restrict our attention to Markov operators that are in a certain sense fully supported. This step is legitimized by the following lemma.

\begin{lemma} \label{One-shotMixing} A discrete measured groupoid $(\mathcal{G},\mu)$ is Choquet--Deny if and only if every invariant Markov operator $P$ on $L^\infty(\mathcal{G}, \mu_s)$ with $\text{supp}(d_{x}P)=\mathcal{G}^{x}$ for every $x\in\mathcal{G}^{(0)}$ is fiberwise Liouville. \end{lemma}

\begin{proof}
The assertion of the ``only if'' direction is trivial. For the ``if'' direction assume that every invariant Markov operator with fully supported transition probabilities is fiberwise Liouville and let $P$ be a non-degenerate invariant Markov operator on $L^{\infty}(\mathcal{G},\mu_{s})$. As discussed in Subsection \ref{MeasuredGroupoids}, for every $n\in\mathbb{N}$ the operator $P^{n}$ is a Markov operator as well. The sequence $(\sum_{i=1}^{n}2^{-i}P^{i})_{n\in\mathbb{N}}$ is furthermore Cauchy and thus norm converges to an element $\widetilde{P}$. The element $\widetilde{P}$ is also a Markov operator, induced by the family $(d_{x}\widetilde{P})_{x\in\mathcal{G}^{(0)}}$ of probability measures given by $d_{x}\widetilde{P}=\sum_{i=1}^{\infty}2^{-i}(d_{x}P^{i})$, and for $x\in\mathcal{G}^{(0)}$ any $P_{x}$-harmonic function must also be $\widetilde{P}_{x}$-harmonic. To conclude the statement of the lemma it hence suffices to show that $\widetilde{P}$ satisfies $\text{supp}(\widetilde{P})=\mathcal{G}^{x}$ for all $x\in\mathcal{G}^{(0)}$. But this is clear since $P$ is non-degenerate.
\end{proof}

\subsection{Equivalence relations}

The goal of this section is to give a complete characterization of the Choquet--Deny property for countable Borel equivalence relations. As a first step, we show in the following proposition that every finite equivalence relation is Choquet--Deny.

\begin{proposition} \label{Prop: finite equivalence relations are CD} Let $\mathcal{R}$ be a finite equivalence relation on a standard probability space $(X,\mu)$. Then $\mathcal{R}$ is Choquet--Deny. \end{proposition} 
\begin{proof}
By Lemma \ref{One-shotMixing} it suffices to restrict to invariant Markov operators $P$ on $L^{\infty}(\mathcal{G},\mu_{s})$ with $\text{supp}(d_{x}P)=\mathcal{G}^{x}$ for every $x\in\mathcal{G}^{(0)}$. Thus, let $P$ be such an operator and assume for $x\in X$ with $\#\mathcal{R}^{x}<\infty$ that $f\in\ell^{\infty}(\mathcal{R}^{x})$ is a function with $P_{x}f=f$. We may choose an element $g\in\mathcal{R}^{x}$ with $f(g)=\max_{h\in\mathcal{R}^{x}}f(h)$. By applying the pigeonhole principle to the equality 
\begin{align*}
f(g)=P_{x}f(g)=\sum_{h\in\mathcal{R}^{s(g)}}(d_{s(g)}P)(h)f(gh)\,,
\end{align*}
we obtain $f(gh)=f(g)$ for all $h\in\mathcal{R}^{s(g)}$. It follows that $f$ is constant and hence $\mathcal{R}$ is Choquet--Deny.
\end{proof}
Recall that for a non-singular action of a countable discrete group $\Gamma$ on a standard probability space $(X,\mu)$, the corresponding transformation groupoid $\Gamma\ltimes(X,\mu)$ is an equivalence relation precisely when the action is essentially free. The following proposition shows that if a transformation groupoid of a (not necessarily free) action is Choquet--Deny then the acting group must be Choquet--Deny as well.

\begin{proposition} \label{Prop: transformation groupoid CD =00003D> group CD} 
Let $\Gamma\curvearrowright(X,\mu)$ be a nonsingular action of a countable discrete group $\Gamma$ on a standard probability space $(X,\mu)$ and assume that the corresponding transformation groupoid is Choquet--Deny. Then $\Gamma$ is Choquet--Deny as well. 
\end{proposition} 

\begin{proof}
Let $\nu$ be a non-degenerate probability measure on $\Gamma$ and write $\mathcal{G}:=\Gamma\ltimes(X,\mu)$. For every $x\in X$ define a probability measure $\pi_{x}$ on $\mathcal{G}^{x}$ by $\pi_{x}(g,g^{-1}\cdot x):=\nu(g)$ for $g\in\Gamma$. Since the support of $\nu$ generates $\Gamma$ as a semigroup (by the discussion in Subsection \ref{MeasuredGroupoids}) the family $(\pi_{x})_{x\in X}$ gives rise to a non-degenerate invariant Markov operator $P$ on $\Gamma\ltimes(X,\mu)$. From the assumption it thus follows that this Markov operator is fiberwise Liouville.

Now let $F\in\ell^{\infty}(\Gamma)$ be a $\nu$-harmonic function, pick an element $x\in X$ for which the Markov operator $P_{x}\colon\ell^{\infty}(\mathcal{G}^{x})\rightarrow\ell^{\infty}(\mathcal{G}^{x})$ is Liouville, and define $H\in\ell^{\infty}(\mathcal{G}^{x})$ by $H(g,g^{-1}\cdot x):=F(g)$ for $g\in\Gamma$. We compute that for all $g\in\Gamma$,

\begin{align*}
P_xH(g, g\inv x) &= \sum_{h\in\Gamma}\nu(h)H\left((g,g^{-1}\cdot x)(h,h^{-1}g^{-1}\cdot x)\right) \\
    &= \sum_{h\in\Gamma}\nu(h)H(gh,h^{-1}g^{-1}\cdot x) = \sum_{h\in\Gamma}\nu(h)F(gh) \\
    &= F(g) = H(g,g^{-1}x)
\end{align*}
so that $H\in H^{\infty}(\mathcal{G}^{x},P_{x})$. Since $P_{x}$ is Liouville, it follows that $H$ is constant. But then also $F$ must be constant. We conclude that $\Gamma$ is Choquet--Deny.
\end{proof}

To prove a converse to Proposition \ref{Prop: finite equivalence relations are CD}, we will need a few auxiliary lemmas. The next result is well known, but we give a proof for completeness. Recall that the group of all permutations of $\mathbb{N}$ that fix all but finitely many elements is denoted by $\mathcal{S}_{\infty}$. It is amenable and icc and therefore not Choquet--Deny by \cite[Theorem~1]{Frisch2019}.

\begin{lemma} \label{Lemma: actions of S_infty} Let $(\mathcal{R},\mu)$ be an ergodic hyperfinite equivalence relation with infinite orbits $\mu$-almost everywhere. Then there exists an essentially free ergodic action $\mathcal{S}_{\infty}\curvearrowright(X,\mu)$ on a a standard probability space $(X,\mu)$ such that the corresponding orbit equivalence relation is isomorphic to $(\mathcal{R},\mu)$. \end{lemma} 
\begin{proof}
Recall that there exist unique ergodic hyperfinite equivalence relations of types $\text{I}_{\infty}$, $\text{II}_{1}$ and $\text{II}_{\infty}$ and the associated Krieger flow for all of them is the translation action $\mathbb{R}\curvearrowright\mathbb{R}$. Furthermore, for any ergodic flow $\mathbb{R}\curvearrowright(Z,\eta)$ that is not isomorphic to the translation flow, there exists a unique type III ergodic hyperfinite equivalence relation with associated Krieger flow $\mathbb{R}\curvearrowright(Z,\eta)$.

First, note that the left translation action of $\mathcal{S}_{\infty}$ on itself is ergodic and essentially free of type I$_{\infty}$. 

Next, for any non-trivial non-atomic standard probability space $(X_{0},\mu_{0})$ the generalized Bernoulli action of $\mathcal{S}_{\infty}$ on $(X,\mu):=\prod_{n\in\mathbb{\mathbb{N}}}(X_{0},\mu_{0})$ defined by $(\sigma\cdot x)_{n}=x_{\sigma^{-1}(n)}$ for $\sigma\in\mathcal{S}_{\infty}$ and $x:=(x_{n})_{n\in\mathbb{N}}\in X$ is ergodic, essentially free and preserves the measure $\mu$. Its orbit equivalence relation is hence of type $\text{II}_{1}$.

Finally, let $\alpha:\mathbb{R}\curvearrowright Z$ be any ergodic non-singular flow and as in \cite[Definition 3.2]{VAES_VERJANS_2023} let $\widehat{\alpha}:\mathbb{R}\curvearrowright\widehat{Z}$ denote its adjoint flow. By \cite[Theoerm 6.1]{VAES_WAHL_2018} there exists an ergodic essentially free non-singular action $\mathcal{S}_{\infty}\curvearrowright(X,\mu)$ with weakly mixing Maharam extension. Thus this Maraham extsion is ergodic and of type $\text{II}_{\infty}$. Let 
\begin{align*}
\omega:\mathcal{S}_{\infty}\times X\rightarrow\mathbb{R},\quad\omega(\sigma,x):=\log\left(\frac{d\sigma^{-1}\mu}{d\mu}(x)\right)
\end{align*}
denote the Radon--Nikodym cocylce associated to $\mathcal{S}_{\infty}\curvearrowright(X,\mu)$ and consider the non-singular action of $\mathcal{S}_{\infty}$ on $X\times\widehat{Z}$ via $\sigma\cdot(x,z):=(\sigma\cdot x,\widehat{\alpha}_{\omega(\sigma,x)}(z))$ for $\sigma\in\mathcal{S}_{\infty}$, $x\in X$, $z\in\widehat{Z}$. Since the action $\mathcal{S}_{\infty}\curvearrowright(X,\mu)$ is essentially free, the one of $\mathcal{S}_{\infty}$ on $X\times\widehat{Z}$ is essentially free as well. By \cite[Proposition 3.4]{VAES_VERJANS_2023} it is thus ergodic and has associated flow $\mathbb{R}\curvearrowright(Z,\eta)$.
\end{proof}

As an immediate corollary of Proposition \ref{Prop: transformation groupoid CD =00003D> group CD} and Lemma \ref{Lemma: actions of S_infty}, we obtain that if an ergodic equivalence relation is Choquet--Deny, it has finite orbits almost everywhere. We prove that the same holds for general equivalence relations, i.e., without assuming ergodicity. We need the following lemma before stating the main theorem.

\begin{lemma} \label{Lemma: disjoint union of CD groupoids is CD} 
For every $i\in\mathbb{N}$, let $(\mathcal{G}_{i},\mu_{i})$ be a discrete measured groupoid with unit space $X_{i}:=\mathcal{G}_{i}^{(0)}$. Consider the disjoint union $X:=\bigsqcup_{i}X_{i}$ equipped with the $\sigma$-finite measure $\mu:=\sum_{i}\mu_{i}$. Then the discrete measured groupoid $(\bigsqcup_{i}\mathcal{G}_{i},\mu)$ is Choquet--Deny if and only if $\mathcal{G}_{i}$ is Choquet--Deny for every $i\in\mathbb{N}$. 
\end{lemma}

\begin{proof}
    Suppose first that $\cG$ is not Choquet--Deny, and let $(\pi_{x})_{x \in X}$ be a Markov operator that is not fiberwise Liouville. Then there is a positive measure set $E \subset X$ such that $\pi_{x}$ is not Liouville for all $x \in E$. Since the union is countable, there exists $i \in \mathbb{N}$ such that $E_{i} = X_{i} \cap E$ is non-null. Hence the Markov operator $(\pi_{x})_{x \in X_{i}}$ is not fiberwise Liouville and $\cG_{i}$ is not Choquet-Deny.
    
 On the other hand if $\cG_{i}$ is not Choquet--Deny for some $i$, then there is a Markov operator $(\pi^{0}_{x})_{x \in X_{i}}$ that is not fiberwise Liouville, and taking any Markov operator $\pi$  on $\cG$ such that $\pi|_{X_{i}} = \pi^{0}$, we see that it is not fiberwise Liouville either.  
\end{proof}

\begin{theorem} \label{Thm: CD if and only if finite orbits} Let $\mathcal{R}$ be a countable Borel equivalence relation on a standard probability space $(X,\mu)$. Then $(\mathcal{R},\mu)$ is Choquet--Deny if and only if $\mu$-almost all $\mathcal{R}$-orbits are finite. \end{theorem} 
\begin{proof}
The ``if'' direction immediately follows from Proposition \ref{Prop: finite equivalence relations are CD}.

For the ``only if'' direction suppose that $\mathcal{R}$ is Choquet--Deny. Since by \cite[Theorem 4.2]{Kaimanovich05} non-amenable equivalence relations are not Choquet--Deny, $\mathcal{R}$ must be hyperfinite. Furthermore, the discussion in Section \ref{Sec: countable Borel equivalence relations} implies that the space $X$ can be partitioned into $\mathcal{R}$-invariant Borel subsets $X_{\text{fin}}$, $X_{\infty}$, and $X^{C}$, such that $\mathcal{R}|_{X_{\text{fin}}}$ has finite orbits, $\mathcal{R}|_{X_{\infty}}$ is of type $\text{I}_{\infty}$ and $\mathcal{R}|_{X^{C}}$ is the non-type I part. Assume that $\mu(X_{\infty}\sqcup X^{C})>0$. By Lemma \ref{Lemma: disjoint union of CD groupoids is CD} it is then enough to show that the restrictions $\mathcal{R}|_{X^{C}}$ and $\mathcal{R}|_{X_{\infty}}$ are not Choquet--Deny. Without loss of generality we may further restrict to the cases where $X=X^{C}$ or $X=X_{\infty}$.

\begin{case}
    \item Assume that $X=X^{C}$ and let $(Z,\eta)$ be a standard probability space for which $L^{\infty}(Z,\eta)$ is isomorphic to the Abelian von Neumann algebra $L^{\infty}(X,\mu)^{\mathcal{R}}$. If we let $\mathcal{R}=\int^{\oplus}\mathcal{R}_{z}d\eta(z)$ be the ergodic decomposition of $\mathcal{R}$ as in Theorem \ref{Thm: ergodic decomposition theorem}, we know that for $\eta$-almost every $z\in Z$ the ergodic equivalence relation $\mathcal{R}_{z}$ is hyperfinite and hence stable. Denoting the unique p.m.p. ergodic hyperfinite equivalence relation by $\mathcal{R}_{0}$, we then see that
    \begin{eqnarray*}
    \mathcal{R}_{0}\times\mathcal{R}&\cong& \mathcal{R}_{0}\times\int_{Z}^{\oplus}\mathcal{R}_{z}d\eta(z) \\
    &\cong& \int_{Z}^{\oplus}(\mathcal{R}_{0}\times\mathcal{R}_{z})d\eta(z) \\
    &\cong& \int_{Z}^{\oplus}\mathcal{R}_{z}d\eta(z) \\
    &\cong& \mathcal{R}\,.
    \end{eqnarray*}
    The main theorem in \cite{Connes_Feldman_Weiss_1981} implies that $\mathcal{R}$ is isomorphic to the orbit equivalence relation of a $\mathbb{Z}$-action on a standard probability space $(Y,\nu)$, i.e., it is generated by a single transformation $T$. Moreover, this $\Z$-action is free, since for any $n \in \N$ the elements $x \in \Fix(n) = \Fix(T^n)$ have an orbit of length $n$, but by assumption $\cR$ has infinite orbits $\mu$-almost everywhere. By Lemma \ref{Lemma: actions of S_infty} we may furthermore pick an essentially free ergodic action $\mathcal{S}_{\infty}\curvearrowright(Z,\eta)$ such that the corresponding orbit equivalence relation is isomorphic to $\mathcal{R}_{0}$. Now we consider the action $\mathcal{S}_{\infty}\times\mathbb{Z}\curvearrowright(Z\times Y,\eta\times\nu)$ and note that this is still an essentially free action whose orbit equivalence relation is isomorphic to $\mathcal{R}_{0}\times\mathcal{R}\cong\mathcal{R}$. But since $\mathcal{R}$ is Choquet--Deny, by Proposition \ref{Prop: transformation groupoid CD =00003D> group CD} the group $\mathcal{S}_{\infty}\times\mathbb{Z}$ must then be Choquet--Deny as well. However, $\mathcal{S}_{\infty}$ is an icc quotient of $\mathcal{S}_{\infty}\times\mathbb{Z}$, hence by the main result of \cite{Frisch2019} $\mathcal{S}_{\infty}\times\mathbb{Z}$ cannot be Choquet--Deny.

    \item Now assume that $X=X_{\infty}$ and notice that in this case $\cR$ is isomorphic to the groupoid $\mathcal{C} \times Z$, where $\mathcal{C}$ is the full equivalence relation on any fixed countable set and $Z$ is a standard Borel space. Consider any infinite non-Choquet--Deny group $G$, for example $G = S_{\infty}$ under the left translation action $G \actson_{lt} G$. Identifying $G$ as the unit space for $\mathcal{C}$ gives us an isomorphism of $\mathcal{C}$ with $G \ltimes_{lt} G$. Thus, $\cR$ is isomorphic to the free action of $G$ on $G \times Z$ given by translation on the first coordinate and the trivial action on the second coordinate. If $\cR$ is Choquet-Deny, then again by Proposition \ref{Prop: transformation groupoid CD =00003D> group CD}, we arrive at a contradiction. \qedhere
\end{case}
\end{proof}

\subsection{Measured fields of groups}

Suppose that $\mathcal{G}$ is a discrete Borel field of groups. That is, $\text{Iso}(\mathcal{G})=\mathcal{G}$. We shall say that $\mathcal{G}$ is a \emph{discrete measured field of groups} if $\mathcal{G}$ is endowed with a discrete measured groupoid structure. The goal of this section is two-fold: first we consider an analogue of FC-hypercentrality for discrete measured fields of groups. After this, we demonstrate that a non-trivial discrete measured field of groups with icc fibers almost everywhere is never Choquet--Deny. For a more elaborate treatment of Borel and measurable fields of Polish groups, we refer the reader to \cite{VaesWouters24}. 

Given a discrete group $\Gamma$, its \emph{FC-center} is the normal subgroup $FC(\Gamma)$ of elements in $\Gamma$ whose conjugacy class is finite. If $\Gamma$ has cardinality $\kappa$, one may recursively construct an increasing sequence of normal subgroups $(FC_{\alpha}(\Gamma))_{\alpha<\kappa^{+}}$ of $\Gamma$ in the following way:
\begin{enumerate}[label={(\arabic*)}]
\item For $\alpha=0$ set $FC_{0}(\Gamma):=\{e\}$.
\item For sucessor ordinals $\alpha+1$ let $FC_{\alpha+1}(\Gamma)$ be the preimage of the FC-center $FC\left(\Gamma/FC_{\alpha}(\Gamma)\right)$ under the quotient map $\Gamma\twoheadrightarrow\Gamma/FC_{\alpha}(\Gamma)$.
\item For limit ordinals $\alpha>0$ set $FC_{\alpha}(\Gamma):=\bigcup_{\beta<\alpha}FC_{\beta}(\Gamma)$.
\end{enumerate}
Note that the sequence $FC_{\alpha}(\Gamma)$ being strictly increasing would imply that $\#\Gamma\geq\kappa^{+}$; thus the sequence eventually stabilizes. We call the least ordinal $\alpha$ for which $FC_{\alpha}(\Gamma)=FC_{\alpha+1}(\Gamma)$ the \emph{FC-rank} of $\Gamma$. The corresponding normal subgroup $FC_{\alpha}(\Gamma)\trianglelefteq\Gamma$ is called the \emph{FC-hypercenter} of $\Gamma$ and we denote it by $FCH(\Gamma)$. The quotient $\Gamma/FCH(\Gamma)$ is an icc group and $\Gamma$ is FC-hypercentral if and only if $FCH(\Gamma)=\Gamma$. Thus, \cite[Theorem 1]{Frisch2019} states that $\Gamma$ is Choquet--Deny precisely when $FC_{\alpha}(\Gamma)=\Gamma$ eventually.

As noted in the remark after \cite[Lemma 7]{mclain1956remarks}, it is possible to construct for each infinite ordinal $\alpha$, a group $\Gamma_{\alpha}$ of cardinality $\#\alpha$ for which the FC-rank of $\Gamma_{\alpha}$ is $\alpha$. Thus, although eventually constant, we have no control over when such a phenomenon must occur.

In \cite{berendschot2024factoriality}, the first four authors define the icc property for discrete measured groupoids. We say that a discrete measured groupoid $(\mathcal{G},\mu)$ is \emph{icc} if for all Borel bisections $A\subseteq\text{Iso}(\mathcal{G}) \setminus \mathcal{G}^{(0)}$, the conjugacy class $\Omega_{A}:=\bigcup_{g\in\mathcal{G}}gAg^{-1}$ has infinite measure. In \cite[Proposition 6.1]{berendschot2024factoriality} it is demonstrated that a discrete measured field of groups $(\mathcal{G},\mu)$ is icc precisely when $\mu$-almost every fiber $\mathcal{G}_{x}$, $x\in \mathcal{G}^{(0)}$ is an icc group.

Our goal is to first define an analogue of the FC-center for Borel fields of groups, and then do a recursion process similar to the one above in order to eventually arrive at an icc groupoid. For this purpose we introduce the following definition.

\begin{definition} Let $\mathcal{G}$ be a Borel field of groups. The \emph{FC-center} of $\mathcal{G}$ is the subgroupoid 
\[
FC(\mathcal{G}):=\{g\in\mathcal{G}\mid\#C_{g}<\infty\}\,,
\]
where $C_{g}$ denotes the conjugacy class of $g$ in $\mathcal{G}_{s(g)}$. \end{definition}

\begin{proposition} \label{FCCenterBorel} 
For a discrete Borel field of groups $\mathcal{G}$ the FC-center $FC(\mathcal{G})$ is Borel. \end{proposition}

\begin{proof}
By the Lusin--Novikov Theorem, there exists a sequence $(\tau_{n})_{n\in\mathbb{N}}$ of Borel sections $\tau_{n}:X_{n}\rightarrow B_{n}$, where $X_{n}\subseteq\mathcal{G}^{(0)}$, $B_{n}\subseteq\mathcal{G}$ and $\mathcal{G}=\bigsqcup_{n\in\mathbb{N}}B_{n}$. Without loss of generality assume that $B_{0}=\mathcal{G}^{(0)}$. Furthermore, note that for any $n\in\mathbb{N}$ the set 
\[
\Omega_{\tau_{n}(X_n)}=\bigsqcup_{x\in X_n}C_{\tau_{n}(x)}=\bigcup_{i\in\mathbb{N}}B_{i}\tau_{n}(X_n)B_{i}^{-1}
\]
is Borel. Since the projection $s|_{\Omega_{\tau_{n}(X_n)}}:\Omega_{\tau_{n}(X_n)}\to \mathcal{G}^{(0)}$ is Borel, it follows from an application of the Lusin--Novikov Theorem that the map $f_{\tau_{n}}:\mathcal{G}^{(0)} \to\mathbb{N}\cup\{\infty\}$ given by $f_{\tau_{n}}(x):=\#(s|_{\Omega_{\tau_{n}(X_n)}})^{-1}(x)$ must Borel as well. But this implies that the set $Y_{n}:=f_{\tau_{n}}^{-1}(\mathbb{N})\subseteq X_{n}$, consisting of those points $x\in\mathcal{G}^{(0)}$ for which $C_{\tau_{n}(x)}$ is finite, is also Borel. Observing that $FC(\mathcal{G})=\bigsqcup_{n \in \mathbb{N}}\tau_{n}(Y_{n})$ completes the proof.
\end{proof}
As a corollary to the proof of Proposition \ref{FCCenterBorel} we have the following result.

\begin{corollary}\label{Corollary: negligible FC} Let $(\mathcal{G},\mu)$ be a discrete measured field of groups and suppose that $\mu_{s}(FC(\mathcal{G})\setminus\mathcal{G}^{(0)})=0$. Then $\mathcal{G}_{x}$ is an icc group for $\mu$-almost every $x\in\mathcal{G}^{(0)}$. \end{corollary} 
\begin{proof}
Define a sequence $(\tau_{n})_{n\in\mathbb{N}}$ of Borel sections and Borel sets $Y_{n}$, $n\in\mathbb{N}$ as in the proof of the previous proposition. Note that by the assumptions on the codomains of the sections, for any $n\geq1$ the intersection of $\tau_{n}(Y_{n})$ and $\mathcal{G}^{(0)}$ is trivial and the sets $\tau_{n}(Y_{n})$ are disjoint bisections. In particular, the assumption $\mu_{s}(FC(\mathcal{G})\setminus\mathcal{G}^{(0)})=0$ implies that $\mu_s(Y_{n})=\mu_{s}(\tau_{n}(Y_{n}))=0$ for all $n\geq1$. It follows that the union $\bigcup_{n\in\mathbb{N}}Y_{n}$ is a negligible set. By construction we have 
\[
\bigcup_{n\in\mathbb{N}}Y_{n}=\{x\in\mathcal{G}^{(0)}\mid\#C_{g}<\infty\text{ for some non-trivial }g\in\mathcal{G}_{x}\}\,.
\]
 Thus, outside of this negligible set, the fibers are all icc groups.
\end{proof}
By \cite[Theorem 3.1]{Sutherland85}, if $\mathcal{G}$ is a Borel field of groups, and $\mathcal{N}\subseteq\mathcal{G}$ is a Borel field of normal subgroups, then the quotient space $\mathcal{G}/\mathcal{N}$ is also a standard Borel space.

We are now ready to define the analogue of $FC_{\alpha}$ for a discrete Borel field of groups.

\begin{definition} \label{HypercenterDefinition} Let $\mathcal{G}$ be a discrete Borel field of groups. Define a $\aleph_{1}$-indexed sequence of quotients $\mathcal{G}\overset{\pi_{\alpha}}{\twoheadrightarrow}\mathcal{G}_{\alpha}$ as well as maps $\pi_{\beta,\alpha}:\mathcal{G}_{\beta}\to\mathcal{G}_{\alpha}$ for $\beta<\alpha<\aleph_{1}$ by using recursion on $\alpha$:
\begin{enumerate}[label=(\arabic*)]
\item For $\alpha=0$ let $\mathcal{G}_{0}:=\mathcal{G}$ and $\pi_{0}:=\text{id}$.
\item If $\alpha<\aleph_{1}$ is a limit ordinal, then set $\mathcal{G}_{\alpha}:=\mathcal{G}/\bigcup_{\beta<\alpha}\ker(\pi_{\beta})$ and let $\pi_{\alpha}$ denote the corresponding quotient map. For each $\beta<\alpha$, since $\ker(\pi_{\beta})\subseteq\ker(\pi_{\alpha})$, the quotient map $\pi_{\alpha}$ factors to a map $\pi_{\beta,\alpha}:\mathcal{G}_{\beta}\to\mathcal{G}_{\alpha}$.
\item If $\alpha=\beta+1$, then set $\mathcal{G}_{\alpha}:=\mathcal{G}_{\beta}/FC(\mathcal{G}_{\beta})$ and let $\pi_{\alpha}$ denote the composition of the map $\pi_{\beta}$ with the quotient map $\pi_{\beta,\alpha}:\mathcal{G}_{\beta}\to\mathcal{G}_{\beta}/FC(\mathcal{G}_{\beta})$. Similarly, for any $\gamma<\beta$, let $\pi_{\gamma,\alpha}:=\pi_{\gamma,\beta}\circ\pi_{\beta,\alpha}$. 
\end{enumerate}
For each $\alpha<\aleph_{1}$ denote the kernel of the map $\pi_{\alpha}$ by $FC_{\alpha}(\mathcal{G})$. \end{definition}

One may ask whether, similar to the group setting, the sequence in Definition \ref{HypercenterDefinition} eventually stabilizes. Although we are not able to prove such a statement, the next proposition illustrates that the sequence stabilizes outside of a negligible set.

\begin{proposition} \label{proposition: eventually negligible} Let $(\mathcal{G},\mu)$ be a discrete measured field of groups. Then there is an ordinal $\alpha<\aleph_{1}$ such that $FC_{\beta}(\mathcal{G})\setminus FC_{\alpha}(\mathcal{G})$ is a negligible set for all $\beta>\alpha$. In particular, $FC(\mathcal{G}_{\alpha})\setminus\mathcal{G}_{\alpha}^{(0)}$ is a negligible set. \end{proposition} 
\begin{proof}
Decompose $\mathcal{G}$ into Borel bisections $\mathcal{G}=\bigsqcup_{n\in\mathbb{N}}B_{n}$. For a fixed integer $n\in\mathbb{N}$, the quantities $\mu_{s}(B_{n}\cap FC_{\alpha}(\mathcal{G}))$ give an increasing sequence of numbers between $0$ and $1$. Let $\alpha_{n}$ be the least ordinal for which $\mu_{s}(B_{n}\cap FC_{\alpha_{n}}(\mathcal{G}))=\mu_{s}(B_{n}\cap FC_{\beta}(\mathcal{G}))$ for all $\beta>\alpha_{n}$. Such an ordinal must exist since otherwise our sequence would be unbounded. As a countable supremum of countable ordinals, $\alpha:=\sup_{n\in\mathbb{N}}\alpha_{n}$ must too be a countable ordinal. For all $\beta>\alpha$ we then have
\[
FC_{\beta}(\mathcal{G})\setminus FC_{\alpha}(\mathcal{G})=\bigsqcup_{n\in\mathbb{N}}B_{n}\cap(FC_{\beta}(\mathcal{G})\setminus FC_{\alpha}(\mathcal{G}))\,,
\]
so that $FC_{\beta}(\mathcal{G})\setminus FC_{\alpha}(\mathcal{G})$ has measure zero.

Since $FC_{\alpha}(\mathcal{G})$ is the kernel of $\pi_{\alpha}$, we get that $FC(\mathcal{G}_{\alpha})=\pi_{\alpha}(FC_{\alpha+1}(\mathcal{G}))=\pi_{\alpha}(FC_{\alpha+1}(\mathcal{G})\setminus FC_{\alpha}(\mathcal{G}))\cup\mathcal{G}_{\alpha}^{(0)}$. It finally follows that $FC(\mathcal{G}_{\alpha})\setminus\mathcal{G}_{\alpha}^{(0)}$ must be negligible.
\end{proof}
The following corollary, along with Corollary \ref{Corollary: negligible FC} demonstrates that eventually $\mathcal{G}/FC_{\alpha}(\mathcal{G})$ is an icc groupoid.

\begin{corollary}\label{Corollary: N_alpha is eventually FCH} Let $(\mathcal{G},\mu)$ be a discrete measured field of groups. Then there exists an $\alpha<\aleph_{1}$ and a co-null Borel subset $E\subseteq\mathcal{G}^{(0)}$ such that 
\[
FC_{\alpha}(\mathcal{G})|_{E}=\bigsqcup_{x\in E}FCH(\mathcal{G}_{x})\,.
\]
In particular, the groupoid $\mathcal{G}_{\alpha}$ is an icc groupoid. \end{corollary}
\begin{proof}
By Proposition \ref{proposition: eventually negligible} we find $\alpha<\aleph_{1}$ for which $FC_{\beta}(\mathcal{G})\setminus FC_{\alpha}(\mathcal{G})$ is negligible for all $\beta>\alpha$. By applying the Lusin--Novikov Theorem, fix a partition $\mathcal{G}=\bigsqcup_{n\in\mathbb{N}}B_{n}$ of Borel bisections $B_{n}$. Since $E_{n}:=s(FC_{\alpha+1}(\mathcal{G})\setminus FC_{\alpha}(\mathcal{G})\cap B_{n})$ is negligible for all $n\in\mathbb{N}$, the set $\bigcup_{n\in\mathbb{N}}E_{n}$ is also negligible. We claim that for every point $x\in E:=\mathcal{G}^{(0)}\setminus(\bigcup_{n\in\mathbb{N}}E_{n})$, we have $FC_{\alpha+1}(\mathcal{G})_{x}=FC_{\alpha}(\mathcal{G})_{x}$. Indeed, the inclusion $FC_{\alpha+1}(\mathcal{G})_{x}\supseteq FC_{\alpha}(\mathcal{G})_{x}$ follows from the construction. Conversely, if $g\in FC_{\alpha+1}(\mathcal{G})_{x}\setminus FC_{\alpha}(\mathcal{G})_{x}$, then there is some $n\in\mathbb{N}$ for which $g\in B_{n}$ and hence $x=s(g)\in E_{n}$, which is a contradiction. We obtain that $FC_{\alpha}(\mathcal{G})_{x}=FCH(\mathcal{G}_{x})$ for all $x\in E$.

Finally, to see that $\mathcal{G}_{\alpha}=\mathcal{G}/FC_{\alpha}(\mathcal{G})$ is icc, first observe that $FC(\mathcal{G}_{\alpha})\setminus\mathcal{G}_{\alpha}^{(0)}$ is a null set. By Proposition \ref{proposition: eventually negligible} and Corollary \ref{Corollary: negligible FC}, $(\mathcal{G}_{\alpha})_{x}$ is an icc group for $\mu$-almost every $x\in\mathcal{G}_{\alpha}^{(0)}$. By \cite[Proposition 6.1]{berendschot2024factoriality}, the discrete measured field of groups $(\mathcal{G}_{\alpha},\mu)$ is icc if and only if $\mu$-almost every $(\mathcal{G}_{\alpha})_{x}$ is icc, hence completing the proof. 
\end{proof}
\begin{remark} As remarked in the preamble to Proposition \ref{proposition: eventually negligible}, we do not currently know whether $FC_{\alpha}(\mathcal{G})$ must eventually stabilize. It would be interesting to know whether this is true since this would allow us to define a definitive notion of the FC-hypercenter of a discrete measured field of groups. \end{remark}

The final goal of this subsection is to show that a non-trivial icc discrete measured field of groups cannot be Choquet--Deny. We do this by exploiting the explicit nature of the proof of \cite[Proposition 2.2]{Frisch2019} in the case of groups. The following lemma generalizes the concept of a switching and super-switching element from \cite[Definition~~2.3]{Frisch2019}.

\begin{lemma}\label{Lemma: Switching elements are Borel} Let $\mathcal{G}$ be a discrete Borel field of groups and let $A\subseteq\mathcal{G}$ be any Borel subset. Define the set of \emph{switching} and \emph{super-switching} elements of $A$ by 
\[
S_{A}:=\{g\in\mathcal{G}\mid gAg^{-1}\cap A\subseteq\mathcal{G}^{(0)}\}
\]
and 
\[
T_{A}:=\{g\in\mathcal{G}\mid(gAg\cup gAg^{-1}\cup g^{-1}Ag\cup g^{-1}Ag^{-1})\cap A\subseteq\mathcal{G}^{(0)}\}
\]
respectively. Then both $S_{A}$ and $T_{A}$ are Borel sets. \end{lemma}
\begin{proof}
By the Lusin--Novikov Theorem, there is a partition $A=\bigsqcup_{n\in\mathbb{N}}A_{n}$ of $A$ such that the restriction $s|_{A_{n}}$ is injective for all $n\in\mathbb{N}$. For a fixed element $n\in\mathbb{N}$ and $x,y\in\{\pm1\}$, observe that the function $f_{n}^{x,y}:\mathcal{G}\to\mathcal{G}$ defined by $f_{n}^{x,y}(g):=g^{x}\cdot(s|_{A_{n}})^{-1}(s(g))\cdot g^{y}$ is Borel. By 
\[
S_{A}=\bigcap_{n \in \mathbb{N}}\{g\in\mathcal{G}\mid f_{n}^{1,-1}(g)\in\mathcal{G}^{(0)}\cup(\mathcal{G}\setminus A)\}
\]
and 
\[
T_{A}=\bigcap_{n,x,y}\{g\in\mathcal{G}\mid f_{n}^{x,y}(g)\in\mathcal{G}^{(0)}\cup(\mathcal{G}\setminus A)\}
\]
it follows that the sets $S_{A}$ and $T_{A}$ are both Borel.
\end{proof}
\begin{definition} Let $\mathcal{G}$ be a discrete Borel field of groups. For a Borel subset $A\subseteq\mathcal{G}$, we say that an element $g\in\mathcal{G}$ is a \emph{switching element} for $A$ if $g\in S_{A}$. Furthermore, We say that $g$ is a \emph{super-switching} element for $A$ if $g\in T_{A}$. \end{definition}

\begin{theorem}\label{theorem: icc field of groups} Suppose that $(\mathcal{G},\mu)$ is a non-trivial icc discrete measured field of groups. Then $\mathcal{G}$ is not Choquet--Deny. \end{theorem} 
\begin{proof}
By the results of \cite{Kaimanovich05, ChuXin2018}, we can assume that all $\mathcal{G}_{x}$, $x\in\mathcal{G}^{(0)}$ are infinite amenable. We proceed by proving the following claim.

\begin{claim*}
    Let $A,B\subseteq\mathcal{G}$ be Borel subsets with finite fibers. Then there exists a Borel section $\sigma:\mathcal{G}^{(0)}\to\mathcal{G}$ such that $\sigma(x)$ is super-switching for $A$ and $\sigma(x)\notin B$ for every $x\in\mathcal{G}^{(0)}$.
\end{claim*}

\begin{proof}[Proof of the Claim] 
Note that since the source map $s:\mathcal{G}\to\mathcal{G}^{(0)}$ is countable-to-one, we can find Borel sections $\sigma_{n}:\mathcal{G}^{(0)}\to\mathcal{G}$ for each $n\in\mathbb{N}$ with $\mathcal{G}=\bigcup_{n\in\mathbb{N}}\sigma_{n}(\mathcal{G}^{(0)})$. By Lemma \ref{Lemma: Switching elements are Borel}, for each $n\in\mathbb{N}$ the set of all elements $x\in\mathcal{G}^{(0)}$ for which $\sigma_{n}(x)$ is super-switching for $A$ is Borel. We can hence define the Borel function $f_{n}:\mathcal{G}^{(0)}\to\mathbb{N}\cup\{\infty\}$ by
\[
f_{n}(x):=\begin{cases}
\infty & \text{, if }\sigma_{n}(x)\text{ is not super-switching for }A\text{ or }\sigma_{n}(x)\in B\\
n & \text{, if }\sigma_{n}(x)\text{ is super-switching for }A\text{ and }\sigma_{n}(x)\notin B
\end{cases}\,.
\]
and let $f:=\inf_{n\in\mathbb{N}}f_{n}$. By \cite[Proposition 2.5]{Frisch2019}, there are infinitely many super-switching elements for each fiber $A_{x}$ in $\mathcal{G}_{x}$. It follows that $f$ is finite. We can then set $\sigma(x)=\sigma_{f(x)}(x)$ to obtain the statement of the claim.
\end{proof}

Let now $\tau:\mathcal{G}^{(0)}\to\mathcal{G}\setminus\mathcal{G}^{(0)}$ be any Borel section. We adapt the construction of \cite[Proposition~2.2]{Frisch2019}. First fix $0<\varepsilon<\frac{1}{8}$ and let $p$, $K,N\in\mathbb{N}$ and $E_{\varepsilon,m}\subseteq\mathbb{N}^{m}$ be the probability measure, the constants, and the events from \cite[Lemma~2.6]{Frisch2019} respectively. Using the Lusin--Novikov Theorem, we may decompose $\mathcal{G}$ via $\mathcal{G}=\bigcup_{n\in\mathbb{N}}\mathcal{G}_{n}$ where the restrictions $s|_{\mathcal{G}_{n}}$, $n\in\mathbb{N}$ are bijective. By shifting and replacing $\mathcal{G}_{n}$ with $\mathcal{G}^{(0)}$ if necessary, we may assume that for every $n\leq N$, $\mathcal{G}_{n}=\mathcal{G}^{(0)}$. We define Borel sections $\tau_{n}:\mathcal{G}^{(0)}\to\mathcal{G}$ and Borel sets $A_{n}$, $B_{n}$, and $C_{n}$ recursively as follows. First set $\tau_{1}=\ldots=\tau_{N}=\text{id}$. Given $\tau_{1},\ldots,\tau_{n}$, set $A_{n}:=\tau_{n}(\mathcal{G}^{(0)})\cup\tau_{n}(\mathcal{G}^{(0)})^{-1}\cup\mathcal{G}_{n}\cup\mathcal{G}_{n}^{-1}$, $B_{n}:=\bigcup_{i\leq n}A_{i}$ and $C_{n}:=B_{n}\cup\tau(\mathcal{G}^{(0)})\cup\tau(\mathcal{G}^{(0)})^{-1}$. Lastly we let $\tau_{n+1}$ be a Borel section of super-switching elements for $(C_{n})^{2n+1}$ that has trivial intersection with $(C_{n})^{8n+1}$. Such a Borel section exists by the claim above.

For each $n\in\mathbb{N}$ define a symmetric probability measure $\mu_{n,x}$ on $(A_n)_x$ by 
\begin{eqnarray*}
\mu_{n,x}&:=& \varepsilon2^{-n}\left(\frac{1}{2}\delta_{s|^{-1}_{\mathcal{G}_{n}}(x)}+\frac{1}{2}\delta_{(s|^{-1}_{\mathcal{G}_{n}}(x))^{-1}}\right) \\
& & \qquad +(1-\varepsilon2^{-n})\left(\frac{1}{2}\delta_{\tau_{n}(x)}+\frac{1}{2}\delta_{\tau_{n}(x)^{-1}}\right)\,.
\end{eqnarray*}
To see that $(\mu_{n,x})_{x\in\mathcal{G}^{(0)}}$ will induce an invariant Markov operator $P_{n}$, it suffices by \cite[Proposition 3.4]{Kaimanovich05} to check that for all non-negative Borel functions $f$ on $\mathcal{G}$, the map $P_{n}f:x\mapsto\int_{\mathcal{G}^{x}}f(g)d\mu_{n,x}(g)$ on $\mathcal{G}^{(0)}$ is Borel. But this easily follows from the identity
\begin{eqnarray*}
P_{n}f(x) & = & \varepsilon2^{-n}\left(\frac{1}{2}f\left(s|_{\mathcal{G}_{n}}^{-1}(x)\right)+\frac{1}{2}f\left((s|_{\mathcal{G}_{n}}^{-1}(x))^{-1}\right)\right)\\
 &  & \qquad+(1-\varepsilon2^{-n})\left(\frac{1}{2}f(\tau_{n}(x))+\frac{1}{2}f\left(\tau_{n}(x)^{-1}\right)\right)\,.
\end{eqnarray*}
Finally, we define $\mu_{x}:=\sum_{n\in\mathbb{N}}p(n)\mu_{n,x}$ for every $x\in\mathcal{G}^{(0)}$. Every such measure has full support since $\text{supp}(p)=\mathbb{N}$ and by \cite[Proposition~2.2]{Frisch2019}, each pair $(\mathcal{G}_{x},\mu_{x})$ is not Liouville. It then follows that the associated Markov operator $\sum_{n}p(n)P_{n}$ is not Lioville, so that $\mathcal{G}$ is not Choquet--Deny. 
\end{proof}
    
\subsection{Proof of Theorem \ref{Main Theorem CD}}

In this subsection, we prove our main result Theorem \ref{Main Theorem CD}. The proof splits up into two parts: in Subsection \ref{ChoquetDenyUnderQuotients} we show that the Choquet--Deny property is inherited by its equivalence relation and by going over to the corresponding isotropy subgroupoid. The converse implication will be treated in Subsection \ref{GroupoidMarkovChains} by introducing suitable Markov chains and by applying Doob's optional stopping formula (see Theorem \ref{DoobTheorem}).

\subsubsection{The Choquet--Deny property under quotients}

\label{ChoquetDenyUnderQuotients}

First we show that under suitable conditions, the Choquet--Deny property passes to quotients.

\begin{proposition} \label{Prop: quotients under finite-to one maps remain CD} 
Let $(\mathcal{G},\mu)$ and $(\mathcal{H},\nu)$ be discrete measured groupoids. Suppose that there is a surjective Borel groupoid homomorphism $\rho:\mathcal{G}\to\mathcal{H}$ such that the restriction $\rho|_{\mathcal{G}^{(0)}}:\mathcal{G}^{(0)}\rightarrow\mathcal{H}^{(0)}$ is measure class preserving and countable-to-one, and suppose that $\rho|_{\mathcal{G}^{x}}:\mathcal{G}^{x}\to\mathcal{H}^{\rho(x)}$ is surjective for $\mu$-almost every $x\in\mathcal{G}^{(0)}$. If $(\mathcal{G},\mu)$ is Choquet--Deny, then $(\mathcal{H},\nu)$ is Choquet--Deny as well. 
\end{proposition} 
\begin{proof}
Since $\rho$ is countable-to-one, the Lusin--Novikov Theorem implies the existence of Borel bijections $\rho_{n}:U_{n}\rightarrow\mathcal{H}$ on Borel subsets $U_{n}\subseteq\mathcal{G}$ such that the map $\rho$ decomposes via $\rho=\bigcup_{n<\kappa}\rho_{n}$ and where $\kappa$ is an at most countable cardinal. Note that the $U_{n}$ need not be mutually disjoint.

\begin{claim*}
    We may assume that for $\mu$-almost every $x\in\mathcal{G}^{(0)}$ and every $n<\kappa$ either $\mathcal{G}^{x}\cap U_{n}=\emptyset$ or the restriction $\rho|_{\mathcal{G}^{x}\cap U_{n}}:\mathcal{G}^{x}\cap U_{n}\to\mathcal{H}^{\rho(x)}$ is bijective.
\end{claim*}

\begin{proof}[Proof of the Claim]
Write $\mathcal{G}^{(0)}=\bigcup_{n<\kappa_{1}}X_{n}$ for suitable Borel subsets $X_{n}\subseteq\mathcal{G}^{(0)}$ for which $\rho|_{X_{n}}:X_{n}\to\mathcal{H}^{(0)}$ is a bijection for every $n$. Note that then the restriction $\rho|_{t^{-1}(X_{n})}:t^{-1}(X_{n})\rightarrow\mathcal{H}$ is still surjective. For each $n$ we may apply the Lusin--Novikov Theorem to find Borel subsets $U_{m}^{n}\subseteq t^{-1}(X_{n})$ with $\bigcup_{m}U_{m}^{n}=t^{-1}(X_{n})$ such that the restriction $\rho|_{U_{m}^{n}}:U_{m}^{n}\to\mathcal{H}$ is a bijection. The subsets $U_{m}^{n}$ then satisfy the required conditions.
\end{proof}

Now suppose that $P$ is an invariant Markov operator on $\mathcal{H}$ with an associated family of probability measures $(d_{y}P)_{y\in\mathcal{H}^{(0)}}$. By Lemma \ref{One-shotMixing} we may assume that $\text{supp}(d_{y}P)=\mathcal{G}^{y}$ for every $y\in\mathcal{H}^{(0)}$. For every $n\in\mathbb{N}$ define a family $(\pi_{x}^{n})_{x\in\mathcal{G}^{(0)}}$of probability measures on $\mathcal{G}$ via
\[
\pi_{x}^{n}:=\begin{cases}
\delta_{x} & \text{, if }x\notin t\circ\sigma_{n}(\mathcal{H})\\
\sigma_{n,\ast}(d_{\rho(x)}P) & \text{, if }x\in t\circ\sigma_{n}(\mathcal{H})
\end{cases}\,,
\]
where $\sigma_{n}:=\rho_{n}^{-1}$ and where $\sigma_{n,\ast}$ denotes the \emph{push forward measure} of $d_{y}P$ with respect to $\sigma_{n}$. Since $\sigma_{n}$ is Borel, it follows that for every positive Borel function $f:\mathcal{G}\to[0,\infty)$ the map $x\mapsto\sum_{g\in\mathcal{G}^{x}}\pi_{x}^{n}(g)f(g)$ is Borel as well. Hence, by \cite[Proposition 3.4]{Kaimanovich05} the family $(\pi_{x}^{n})_{x\in\mathcal{G}^{(0)}}$ induces an invariant Markov operator $Q^{n}$ on $L^{\infty}(\mathcal{G},\mu_{s})$.

Given an element $y\in\mathcal{H}^{(0)}$, suppose that $f\in H^{\infty}(H^{y},P_{y})$. Furthermore, let $x:=t(\sigma_{n}(y))$ and let $g\in\mathcal{G}$ be an element with $s(g)=x$, so that $s(\rho(g))=y$. Then, 
\begin{align}
Q^{n}(f\circ\rho)(g) & =\sum_{h\in\mathcal{G}^{s(g)}}\pi_{s(g)}^{n}(h)(f\circ\rho)(gh)\nonumber \\
 & =\sum_{h\in\mathcal{G}^{s(g)}}\left(\sigma_{n,\ast}(d_{y}P)(h)\right)f\left(\rho(g)\rho(h)\right)\nonumber \\
 & =\sum_{h\in\mathcal{G}^{s(g)}\cap U_{n}}\left(d_{y}P(\rho_{n}(h))\right)f(\rho(g)\rho_{n}(h))\nonumber \\
 & =\sum_{k\in\mathcal{H}^{y}}\left(d_{y}P(k)\right)f(\rho(g)k)\nonumber \\
 & =(f\circ\rho)(g)\label{eq:HarmonicityCalculation}
\end{align}

Now fix a probability measure $\theta$ on the ordinal $\kappa$ with full support and consider the Markov operator $Q$ on $\mathcal{G}$ induced by $d_{x}Q:=\sum_{n<\kappa}\theta(n)\left(d_{x}Q^{n}\right)$ for $x\in\mathcal{G}^{(0)}$. Since $P$ satisfies $\text{supp}(d_{y}P)=\mathcal{G}^{y}$ for every $y\in\mathcal{H}^{(0)}$ and $\rho|_{\mathcal{G}^{(0)}}$ is measure class preserving, $Q$ also satisfies $\text{supp}(d_{x}Q)=\mathcal{G}^{y}$ for every $x\in\mathcal{G}^{(0)}$. Indeed, if $g\in\mathcal{G}^{x}$, then $g\in U_{n}$ for some $n$ and hence $d_{x}Q^{n}(g)=(d_{\rho(x)}P)(\rho(g))\neq0$. The calculation in \eqref{eq:HarmonicityCalculation} furthermore shows that $f\circ(\rho|_{\mathcal{G}^{x}})\in H^{\infty}(\mathcal{G}^{x},Q_{x})$ for every $x\in\rho^{-1}(y)$. It follows that $f\circ\rho$ is constant.

Finally, suppose there is some non-negligible subset $E\subseteq\mathcal{H}^{(0)}$ such that for all $y\in E$ there exists a non-constant function $f_{y}\in H^{\infty}(\mathcal{H}^{y},P_{y})$. Then, by the above we get that $f_{y}\circ(\rho|_{\mathcal{G}^{x}})\in H^{\infty}(\mathcal{G}^{x},Q_{x})$ for each $x\in\rho^{-1}(E)$. Since $\rho|_{\mathcal{G}^{(0)}}$ is measure class preserving, the set $\rho^{-1}(E)$ is non-negligible in $\mathcal{G}^{(0)}$, so the statement of the proposition follows by contraposition. 
\end{proof}

As a first immediate corollary of Proposition \ref{Prop: quotients under finite-to one maps remain CD}, we deduce that the Choquet--Deny property of a groupoid is inherited by its equivalence relation.

\begin{corollary}\label{Corollary: CD passes to eq rel} Suppose that $(\mathcal{G},\mu)$ is a Choquet--Deny discrete measured groupoid. Then its associated equivalence relation is Choquet--deny and hence has finite orbits $\mu$-almost everywhere. \end{corollary} 
\begin{proof}
Let $\mathcal{R}$ be the equivalence relation associated with $(\mathcal{G},\mu)$. Note that the map $G\to\mathcal{R}$ defined by $g\mapsto(t(g),s(g))$ is a surjective groupoid homomorphism. Hence, by Proposition \ref{Prop: quotients under finite-to one maps remain CD} the equivalence relation $\mathcal{R}$ is Choquet--Deny. The finiteness condition follows from Theorem \ref{Thm: CD if and only if finite orbits}.
\end{proof}
As a second implication of Proposition \ref{Prop: quotients under finite-to one maps remain CD}, we obtain a describtion of when a discrete measured field of groups is Choquet--Deny.

\begin{corollary}\label{Corollary: bundle of groups CD} A discrete measured field of groups $(\mathcal{G},\mu)$ is Choquet--Deny if and only if for $\mu$-almost every $x\in\mathcal{G}^{(0)}$ the isotropy group $\mathcal{G}_{x}$ is Choquet--Deny. \end{corollary} 
\begin{proof}
First assume that there exists a non-negligible subset $E\subseteq\mathcal{G}^{(0)}$ such that $\mathcal{G}^{x}$ is not Choquet--Deny for all $x\in E$. Let $\alpha<\aleph_{1}$ be such that $\mathcal{G}/FC_{\alpha}(\mathcal{G})$ is an icc groupoid (see Corollary \ref{Corollary: N_alpha is eventually FCH}) and suppose that $F\subseteq\mathcal{G}^{(0)}$ is a co-null set for which $FC_{\alpha}(\mathcal{G})_{x}=FCH(\mathcal{G}_{x})$ for all $x\in F$. Then for all $x\in E\cap F$ the quotient $(\mathcal{G}/FC_{\alpha}(\mathcal{G}))_{x}$ is non-trivial and in particular not Choquet--Deny. With Proposition \ref{Prop: quotients under finite-to one maps remain CD}, it thus follows that $\mathcal{G}$ is not Choquet--Deny.

Conversely, suppose that there is a co-null set $F\subseteq\mathcal{G}^{(0)}$ for which $\mathcal{G}_{x}$ is Choquet--Deny for all $x\in F$ and let $(\pi_{g})_{g\in G}\subseteq\text{Prob}(\mathcal{G})$ be any invariant Borel family of probability measures on $\mathcal{G}$. By assumption, $\pi_{x}\in\text{Prob}(\mathcal{G}_{x})$ is Liouville for all $x\in E$. Thus, the family $(\pi_{g})_{g\in G}$ is fiberwise Liouville.
\end{proof}
As a third corollary, we deduce that the Choquet--Deny property is also inherited by the isotropy subgroupoid.

\begin{corollary}\label{Corollary: CD passes to Iso} Suppose that $(\mathcal{G},\mu)$ is a Choquet--Deny discrete measured groupoid. Then $\text{Iso}(\mathcal{G})$ is Choquet--Deny, and hence $\mathcal{G}_{x}^{x}$ is Choquet--Deny for $\mu$-almost every $x\in\mathcal{G}^{(0)}$. \end{corollary} 
\begin{proof}
Denote the equivalence relation associated with $\mathcal{G}$ by $\mathcal{R}$. By Corollary \ref{Corollary: CD passes to eq rel}, it follows that $\mathcal{R}$ has finite orbits $\mu$-almost everywhere. We may assume that $\mathcal{G}$ is amenable and by Proposition \ref{Proposition: semi-direct product groupoid} we can write $\mathcal{G}$ as a semi-direct product of the form $\mathcal{G}=\Gamma\rtimes_{\delta}\mathcal{R}$ where $\Gamma=(\mathcal{G}_{x}^{x})_{x\in\mathcal{G}^{(0)}}$ is the isotropy bundle of $\mathcal{G}$. Since $\mathcal{R}$ is of type I, we can find a Borel selector $E:\mathcal{G}^{(0)}\to\mathcal{G}^{(0)}$ as in Theorem \ref{Theorem: type I eq rel}. Let $Y:=E(\mathcal{G}^{(0)})$ be the fundamental domain. Since $Y$ is Borel, the restriction $\Gamma|_{Y}$ is still a discrete Borel field of groups. We construct a quotient map $\pi:\mathcal{G}\to\Gamma|_{Y}$ by setting $\pi(g,(y,x))=\delta_{(E(x),x)}(g)$ where $g\in\Gamma$ and $(y,x)\in\mathcal{R}$. It is clear that $\pi$ is surjective and the restriction of $\pi$ to $\mathcal{G}^{(0)}$ is a finite-to-one map $\mathcal{G}^{(0)}\to Y$. Also, $\pi$ is surjective when restricted to target fibers. We show that $\pi$ is a groupoid homomorphism. To see that multiplication is preserved, we calculate 
\begin{align*}
\pi\left((h,(z,y))(g,(y,x))\right) & =\pi\left((\delta_{(x,y)}(h)\cdot g,(z,x))\right)\\
 & =\delta_{(E(x),x)}(\delta_{(x,y)}(h)\cdot g)\\
 & =\delta_{(E(x),y)}(h)\cdot\delta_{(E(x),x)}(g)\\
 & =\delta_{(E(y),y)}(h)\cdot\delta_{(E(x),x)}(g)\\
 & =\pi\left((h,(z,y))\right)\cdot\pi\left((g,(y,x))\right)
\end{align*}
for $g,h\in\Gamma$, $x,y,z\in\mathcal{G}^{(0)}$. Similarly, we see that the inverse is preserved via
\begin{align*}
\pi\left((g,(y,x))^{-1}\right) & =\pi\left((\delta_{(y,x)}(g^{-1}),(x,y))\right)\\
 & =\delta_{(E(y),y)}((\delta_{(y,x)}(g^{-1}))\\
 & =\delta_{(E(y),x)}(g^{-1})\\
 & =\delta_{(E(x),x)}(g)^{-1} \\
 & =\pi\left((g,(y,x))\right)^{-1} \,.
\end{align*}
Hence, by Proposition \ref{Prop: quotients under finite-to one maps remain CD} it follows that $\Gamma|_{Y}$ is Choquet--Deny. By Corollary \ref{Corollary: bundle of groups CD} it furthermore follows that $\Gamma_{y}$ is Choquet--Deny for $\mu$-almost every $y\in Y$. Since $Y$ is a fundamental domain for $\mathcal{G}^{(0)}$, we obtain that $\mathcal{G}_{x}^{x}$ is Choquet--Deny for $\mu$-almost every $x\in\mathcal{G}^{(0)}$.
\end{proof}
Note that Corollary \ref{Corollary: CD passes to eq rel} and Corollary \ref{Corollary: CD passes to Iso} together imply one direction of Theorem \ref{Main Theorem CD}.

Finally, we can exactly describe the icc groupoids which are Choquet--Deny: they are exactly the finite equivalence relations. The proof of this characterization requires the following lemma.

\begin{lemma} \label{Lemma:IsotropyICC}Let $(\mathcal{G},\mu)$ be an icc discrete measured groupoid and assume that its associated countable Borel equivalence relation has finite orbits $\mu$-almost everywhere. Then the isotropy subgroupoid $\text{Iso}(\mathcal{G})$ is also icc. \end{lemma}
\begin{proof}
We argue by contraposition. Suppose that $(\mathcal{G},\mu)$ is a discrete measured groupoid for which the associated countable Borel equivalence relation has finite orbits $\mu$-almost everyhwere. Furthermore, let  $X\subseteq\mathcal{G}^{(0)}$ be a non-negligable Borel subset such that for every $x\in X$ the isotropy group $\mathcal{G}_{x}^{x}$ is not icc. By restricting $X$ if necessary, we can assume that there is bound both on the size of the equivalence classes of $x\in X$, as well as on the size of the smallest non-trivial finite conjugacy class in $\mathcal{G}_{x}^{x}$. Choose a section $\rho:X\to\text{Iso}(\mathcal{G}|_{X})$ such that $\rho(x)$ is non-trivial and has finite conjugacy class of minimal size in $\mathcal{G}_{x}^{x}$. Define $A:=\rho(X)$. Then $A\cap\mathcal{G}^{(0)}=\emptyset$ and $\mu_{s}\left(\bigcup_{g\in\mathcal{G}}gAg^{-1}\right)<\infty$, so that $\mathcal{G}$ is not icc.
\end{proof}
\begin{corollary}\label{corollary: icc CD groupoids} An icc discrete measured groupoid is Choquet--Deny if and only if it is a finite equivalence relation.
\end{corollary} 
\begin{proof}
Suppose that $(\mathcal{G},\mu)$ is an icc discrete measured groupoid that is Choquet--Deny. By Corollary \ref{Corollary: CD passes to eq rel} its equivalence relation is Choquet--Deny and hence has finite orbits $\mu$-almost everywhere. It follows from Lemma \ref{Lemma:IsotropyICC} that $\text{Iso}(\mathcal{G})$ is also icc, i.e., $\mathcal{G}_{x}^{x}$ is an icc group for $\mu$-almost every $x\in\mathcal{G}^{(0)}$. Since $\text{Iso}(\mathcal{G})$ is also Choquet--Deny by Corollary \ref{Corollary: CD passes to Iso}, it then follows from Theorem \ref{theorem: icc field of groups} that $\mathcal{G}_{x}^{x}$ is trivial for $\mu$-almost every $x\in\mathcal{G}^{(0)}$. But this precisely means that $(\mathcal{G},\mu)$ is an equivalence relation. 
\end{proof}

\subsubsection{Markov chains from measured groupoids} \label{GroupoidMarkovChains}

For this subsection, we fix a countable Borel groupoid $\mathcal{G}$ and an invariant Markov operator $P$ on $\mathcal{G}$ that is induced by a Borel family of probability measures $(\pi_{g})_{g\in\mathcal{G}}$. For fixed $x\in\mathcal{G}^{(0)}$ define a family of $\sigma$-algebras $(\mathcal{F}_{i})_{i\in\mathbb{N}}$ on $\Omega:=(\mathcal{G}^{x})^{\mathbb{N}}$ via 
\[
\mathcal{F}_{i}:=\sigma\left(\{\omega\in\Omega\mid \omega_{0}=g_{0},\ldots,\omega_{i}=g_{i}\}\mid g_{0},\ldots,g_{i}\in\mathcal{G}^{x}\right)
\]
and set $\mathcal{F}:=\sigma(\bigcup_{i\in\mathbb{N}}\mathcal{F}_{i})$. Note that $\mathcal{F}_{0}\subseteq\mathcal{F}_{1}\subseteq\ldots\subseteq\mathcal{F}$, that is $(\mathcal{F}_{i})_{i\in\mathbb{N}}$ is a filtration. As $\mathcal{G}^{x}$ is countable, by the Daniell-Kolmogorov Extension Theorem \cite[Theorem 1.2.1]{Yadin23}, any probability measure on the measurable space $(\Omega,\mathcal{F})$ is uniquely determined by its values on the \emph{cylinder sets} $\mathcal{Z}(g_0,\ldots, g_i) := \{\omega\in\Omega\mid \omega_{0}=g_{0},\ldots,\omega_{i}=g_{i}\}$ for $i\in\mathbb{N}$ and $g_{0},\ldots,g_{i}\in\mathcal{G}^{x}$. We equip $(\Omega,\mathcal{F})$ with the probability measure $\mu_x$ given by 
\[
\mu_x(\mathcal{Z}(g_0,\ldots,g_i)):=\delta_{x,g_{0}}\prod_{j=0}^{i-1}\pi_{s(g_{j})}(g_{j}^{-1}g_{j+1})\,.
\]
and define a Markov chain $(X_{i})_{i\in\mathbb{N}}$ with transition matrix 
\[
P_{x}:=(\pi_{s(g)}(g^{-1}h))_{g,h\in\mathcal{G}^{x}}
\]
by the projection maps $X_{i}:(\Omega,\mathcal{F},\mu_{x})\rightarrow\mathcal{G}^{x}$, $ \omega \mapsto \omega_{i}$.

The following statement immediately follows from the definitions.

\begin{lemma} \label{IrreducibleEquivalence} 
Let $x \in \cG^{(0)}$. The Markov operator $P_x$ defined as above is irreducible if and only if the Markov chain $(X_{i})_{i\in\mathbb{N}}$ is irreducible.
\end{lemma}

As in the group setting, the notion of martingales is intimately connected to that of harmonic functions (see e.g. \cite[Exercise 2.16]{Yadin23}), as the following proposition illustrates.

\begin{proposition} \label{HarmonicMartingale} 
Assume that the Markov operator $P$ is non-degenerate and let $x\in\mathcal{G}^{(0)}$. A function $f\in\ell^{\infty}(\mathcal{G}^{x})$ is contained in $H^{\infty}(\mathcal{G}^{x},P_{x})$ (i.e., satisfies $P_{x}f=f$) if and only if $(f\circ X_{i})_{i\in\mathbb{N}}$ is a martingale with respect to the filtration $(\mathcal{F}_{i})_{i\in\mathbb{N}}$ as above. \end{proposition}
\begin{proof}
First note that for $f\in\ell^{\infty}(\mathcal{G}^{x})$, the family $(f\circ X_{i})_{i\in\mathbb{N}}$ is adapted to the filtration $(\mathcal{F}_{i})_{i\in\mathbb{N}}$. For the ``if'' direction suppose that $(f\circ X_{i})_{i\in\mathbb{N}}$ is a martingale with respect to $(\mathcal{F}_{i})_{i\in\mathbb{N}}$, i.e., $\mathbb{E}_{\mu_{z}}[|f\circ X_{i}|]<\infty$ and $\mathbb{E}_{\mu_{z}}[f\circ X_{i+1}\mid\mathcal{F}_{i}]=f\circ X_{i}$ $\mu_{x}$-almost surely for all $i\in\mathbb{N}$. Then, 
\[
\mathbb{E}_{\mu_{x}}[(f\circ X_{i+1})\cdot1_{F}]=\mathbb{E}_{\mu_{x}}[\mathbb{E}_{\mu_{x}}[f\circ X_{i+1}\mid\mathcal{F}_{i}]\cdot1_{F}]=\mathbb{E}_{\mu_{x}}[(f\circ X_{i})\cdot1_{F}]
\]
for every $F\in\mathcal{F}_{i}$. Since $P$ is non-degenerate, for fixed $g\in G$ we find $i\geq1$ with $\mu_{x}(F_{g}^{(i)})>0$, where $F_{g}^{(i)}:=\{\omega\in\Omega\mid \omega_{i}=g\}\in\mathcal{F}_{i}$. We thus obtain 
\begin{align*}
f(g) & = \frac{1}{\mu_{x}(F_{g}^{(i)})}\int_{F_{g}^{(i)}}f\circ X_{i}\;d\mu_{x}\\
 & = \frac{1}{\mu_{x}(F_{g}^{(i)})}\int_{F_{g}^{(i)}}f\circ X_{i+1}\;d\mu_{x}\\
 & = \frac{1}{\mu_{x}(F_{g}^{(i)})}\sum_{h\in\mathcal{G}^{x}}\mu_{x}(F_{g}^{(i)} \cap F_{h}^{(i+1)})f(h)\\
 & = \sum_{h\in\mathcal{G}^{x}}\pi_{s(g)}(g^{-1}h)f(h)\,.
\end{align*}
It follows that $f\in H^{\infty}(\mathcal{G}^{x},P_{x})$.

For the ``only if'' direction assume that $f\in H^{\infty}(\mathcal{G}^{x},P_{x})$. For any $i\in\mathbb{N}$, $g\in \cG^x$ we have 
\[
\mathbb{E}_{\mu_{x}}[|f\circ X_{i}|]=\int_{\Omega}|f\circ X_{i}|\;d\mu_{x}=\sum_{g\in \cG^x}\mu_{x}(F_{g}^{(i)})|f(g)|<\infty\,,
\]
where $F_{g}^{(i)}$ is defined as before. It remains to show that for all $i \in \mathbb{N}$, $\mathbb{E}_{\mu_{x}}[f\circ X_{i+1}\mid\mathcal{F}_{i}]=f\circ X_{i}$ $\mu_{x}$-almost surely. Indeed, for $F\in\mathcal{F}_{i}$, 
\begin{align*}
\mathbb{E}_{\mu_{x}}[(f\circ X_{i+1})\cdot1_{F}]  &= \int_{F}f\circ X_{i+1} \; d\mu_{x}\\
  &= \sum_{g\in\mathcal{G}^{x}}\mu_{x}(F\cap F_{g}^{(i+1)})f(g)\\
  &= \sum_{g\in\mathcal{G}^{x}}\sum_{h\in\mathcal{G}^{x}}\mu_{x}(F\cap F_{h}^{(i)}\cap F_{g}^{(i+1)})f(g) \,. 
\end{align*}
Note that as $F$ is the disjoint union of cylinder sets in $\mathcal{F}_i$, we have the identity 
\begin{align*}
    \mu_{x}(F\cap F_{h}^{(i)}\cap F_{g}^{(i+1)}) = \mu_{x}(F\cap F_{h}^{(i)})\pi_{s(h)}(h^{-1}g)\,.
\end{align*}
It therefore follows that 
\begin{align*} 
\mathbb{E}_{\mu_{x}}[(f\circ X_{i+1})\cdot1_{F}]   &= \sum_{g\in\mathcal{G}^{x}}\sum_{h\in\mathcal{G}^{x}}\mu_{x}(F\cap F_{h}^{(i)})\pi_{s(h)}(h^{-1}g)f(g)\\
  &= \sum_{h\in\mathcal{G}^{x}}\mu_{x}(F\cap F_{h}^{(i)})f(h)\\
  &= \int_{F}f\circ X_{i} \;d\mu_{x}\\
  &= \mathbb{E}_{\mu_{x}}[(f\circ X_{i})\cdot1_{F}]\,.\qedhere
\end{align*}
\end{proof}
For fixed $x\in\mathcal{G}^{(0)}$ assume that $\#s(\mathcal{G}^{x})<\infty$ and define the \emph{return time} $T$ to $\cG_x^x$ by 
\[
T: (\Omega, \mathcal{F}, \mu_x) \to \mathbb{N} \cup \{\infty\},\; \omega \mapsto \inf\{i\geq1\mid X_i(\omega) \in\mathcal{G}_{x}^{x}\}\;.
\]

\begin{lemma} \label{ReturnTime} Let $\cG$ be a countable Borel groupoid and let $x \in \cG^{(0)}$ be such that $\#s(\cG^x) < \infty$. The return time $T$ to $\cG_x^x$ defines a stopping time for $(\mathcal{F}_{i})_{i\in\mathbb{N}}$ such that $T< \infty$ $\mu$-almost surely. \end{lemma}
\begin{proof}
For $i\in\mathbb{N}$ we have that 
\begin{align*}
T^{-1}(\{i\}) & = \{\omega\in\Omega\mid X_{0}(\omega),\ldots,X_{i-1}(\omega)\notin\mathcal{G}_{x}^{x}\text{ and }X_{i}(\omega)\in\mathcal{G}_{x}^{x}\}\\
 & = \bigsqcup_{(g_{0},\ldots,g_{i-1})\in(\mathcal{G}^{x}\setminus\mathcal{G}_{x}^{x})^{i},g_{i}\in\mathcal{G}_{x}^{x}}\mathcal{Z}(g_0,\ldots, g_i)\\
 & \in \mathcal{F}_{i}\,,
\end{align*}
so that $T$ is indeed a stopping time. To prove that $T < \infty$ $\mu$-almost surely, fix $k\in\mathbb{N}$ and $K>k$. We have 
\begin{align*}
 \mu_{x}(\{\omega \in\Omega\mid T(\omega)>K\})
 & =  \sum_{g_{0},\ldots,g_{K}\in\mathcal{G}^{x}\setminus\mathcal{G}_{x}^{x}}\mu_{x}(\mathcal{Z}(g_0,\ldots,g_K))\,. 
 \end{align*} 
Note by definition of the measures $(\mu_z)_{z\in\cG^{(0)}}$ that we have the decomposition 
\begin{align*}
     \mu_x(\mathcal{Z}(g_0,\ldots,g_K)) = \mu_x(\mathcal{Z}(g_0,\ldots, g_k))\mu_{s(g_k)}(\mathcal{Z}(g_k^{-1}g_k,\ldots, g_k^{-1}g_K))
\end{align*}
for all $g_0,\ldots, g_K \in \cG^x$. In particular, if $g_i \not\in \cG_x^x$ for all $i$ then $g_k^{-1}g_j \not\in \cG^{s(g_k)}_x$ for all $j$. For $n \geq 1$ and $z \in \cG^{(0)}$, define the remainder term
\begin{align*}
    R_n(z) &= \mu_z(\{\omega \in (\cG^z)^\mathbb{N} \mid \omega_0,\ldots, \omega_n \not\in \cG^z_x\}) \\
    &= \sum_{h_0,\ldots, h_n \in \cG^z \setminus \cG_x^z} \mu_z(\mathcal{Z}(h_0,\ldots,h_n)) \,.
\end{align*}
Rearranging the summation gives us the bound 
\begin{align*}
    \mu_x(\{\omega \in \Omega \mid T(\omega) > K\}) &= \sum_{g_0,\ldots, g_{K-k} \in \cG^x \setminus \cG_x^x}  \mu_x(\mathcal{Z}(g_0,\ldots,g_k)) \cdot R_{k}(s(g_{K-k}))\\
    &\leq \sum_{g_0,\ldots, g_{K-k} \in \cG^x \setminus \cG_x^x} \mu_x(\mathcal{Z}(g_0,\ldots,g_k)) \cdot \alpha_k  
\end{align*}
where $\alpha_k := \max_{z \in s(\cG^x)} R_k(z)$. This establishes the inequality 
\begin{align*}
    \mu_x(\{\omega \in \Omega \mid T(\omega) > K \}) \leq \mu_x(\{\omega \in \Omega \mid T(\omega) > k\}) \cdot \alpha_k \,. 
\end{align*}
Since the above inequality holds for arbitrary $K > k$, it follows by induction on $n$ that 
\begin{align*}
    \mu_x(\{\omega \in \Omega \mid T(\omega) > 2^nk\}) \leq \alpha_k^{n+1} 
\end{align*}
for all $n,k \in \mathbb{N}$. By irreducibility of $P$, for each $z \in s(\cG^x)$, there must be some $k$ for which have $R_k(z) < 1$. Since $\# s(\cG^x) < \infty$, we can choose $k$ large enough so that we have the bound $\alpha_k <1$. Letting $n$ tend to infinity tells us $T < \infty$ $\mu$-almost surely. 
\end{proof}
\begin{proposition}
\label{RestrictionHarmonic} 
For every function $f\in H^{\infty}(\mathcal{G}^{x},P_{x})$ the restriction to $\mathcal{G}_{x}^{x}$ is $\nu$-harmonic, where $\nu\in\text{Prob}(\mathcal{G}_{x}^{x})$ is the \emph{hitting measure} given by 
\[
\nu(g):=\mu_{x}(\{\omega\in\Omega \mid X_T(\omega)=g\})
\]
for $g\in\mathcal{G}_{x}^{x}$. 
\end{proposition}

\begin{proof}
Let $f\in H^{\infty}(\mathcal{G}^{x},P_{x})$. As we have seen, $T$ defines a stopping time for $(\mathcal{F}_{i})_{i\in\mathbb{N}}$ such that $T < \infty$ $\mu$-almost surely and $(f\circ X_{i})_{i\in\mathbb{N}}$ is a martingale. By the boundedness of $f$, the family $(f\circ X_{i})_{i\in\mathbb{N}}$ is uniformly integrable with $\mathbb{E}_{\mu_{x}}[|f\circ X_T|]<\infty$. We can therefore apply Doob's optional stopping formula in Theorem \ref{DoobTheorem} to obtain 
\begin{eqnarray}
f(x) & = & \int_{\Omega}f\left(X_T(\omega)\right)d\mu_{x}(\omega)\nonumber \\
 & = & \sum_{h\in\mathcal{G}_{x}^{x}}\mu_{x}(\{\omega\in\Omega\mid X_T(\omega)=h\})f(h)\nonumber \\
 & = & \sum_{h\in\mathcal{G}_{x}^{x}}\nu(h)f(h)\,.\label{eq:IdentityAtX}
\end{eqnarray}
Since the space $H^{\infty}(\mathcal{G}^{x},P_{x})$ is invariant under the natural action of $\mathcal{G}_{x}^{x}$, we deduce with Equation (\ref{eq:IdentityAtX}) that 
\[
f(g)=(g^{-1}\cdot f)(x)=\sum_{h\in\mathcal{G}_{x}^{x}}\nu(h)(g^{-1}\cdot f)(h)=\sum_{h\in\mathcal{G}_{x}^{x}}\nu(h)f(gh)
\]
for all $f\in H^{\infty}(\mathcal{G}^{x},P_{x})$, $g\in\mathcal{G}_{x}^{x}$. This implies that for every function $f\in H^{\infty}(\mathcal{G}^{x},P_{x})$ the restriction to $\mathcal{G}_{x}^{x}$ is $\nu$-harmonic. 
\end{proof}
\begin{lemma} \label{HittingMeasureIrreducible} Assume that the Markov operator $P$ is non-degenerate. Then the hitting measure $\nu\in\text{Prob}(\mathcal{G}_{x}^{x})$ defined in Proposition \ref{RestrictionHarmonic} is also non-degenerate. \end{lemma}
\begin{proof}
It suffices to show that every element $g\in\mathcal{G}_{x}^{x}$ can be written as a product of elements in the support of $\nu$. Since $P$ is non-degenerate, by Lemma \ref{IrreducibleEquivalence} we can find a sequence $g_{0}=x,\ldots,g_{i}=g$ in $\mathcal{G}^{x}$ with 
\[
\mu_{x}(\mathcal{Z}(g_0,\ldots,g_i))>0\,.
\]
Define $J:=\{0\leq j\leq i\mid g_{j}\in\mathcal{G}_{x}^{x}\}$ and let $0=j_{0}<j_{1}<\ldots<j_{k}=i$ be the elements in $J$. For every $1\leq l\leq k$ set $u_{l}:=g_{j_{l-1}}^{-1}g_{j_{l}}\in\mathcal{G}_{x}^{x}$. Then, $g=u_{1}\ldots u_{k}$ and 
\begin{align*}
\nu(u_{l}) & =  \mu_{x}(\{\omega\in\Omega\mid X_T(\omega)=u_{l}\})\\
 & \geq \mu_{x}(\mathcal{Z}(x,g_{j_{l-1}}^{-1}g_{j_{l-1}+1},\ldots, g_{j_{l-1}}^{-1}g_{j_l})) \\
 & =  \mu_x(\mathcal{Z}(g_{j_{l-1}},g_{j_{l-1}+1},\ldots, g_{j_l}))\\
 & \geq \mu_x(\mathcal{Z}(g_0,\ldots,g_{j_{l-1}-1}))\mu_x(\mathcal{Z}(g_{j_{l-1}},g_{j_{l-1}+1},\ldots, g_{j_l}))\mu_x(\mathcal{Z}(g_{j_l+1},\ldots,g_i)) \\
 &= \mu_{x}\left(\mathcal{Z}(g_0,\ldots,g_i) \right)\\
 & >  0\,.
\end{align*}
It follows that $\nu$ is non-degenerate. 
\end{proof}
The discussion above implies the following statement, which gives the remaining implication in the characterization of Choquet--Deny groupoids in Theorem \ref{Main Theorem CD}. The proof is inspired by \cite[Proposition 3.4]{MeyerovitchYadin16} (see also \cite[Theorem 3.9.7]{Yadin23}).

\begin{theorem} \label{IsotropyToGroupoid} Let $(\mathcal{G},\mu)$ be a discrete measured groupoid. If the countable Borel equivalence relation associated with $(\mathcal{G},\mu)$ has finite orbits $\mu$-almost everywhere and if $\mu$-almost all isotropy groups are Choquet--Deny, then $(\mathcal{G},\mu)$ is Choquet--Deny as well. \end{theorem}
\begin{proof}
Let $x\in\mathcal{G}^{(0)}$ be an element for which $\mathcal{G}_{x}^{x}$ is Choquet--Deny and whose orbit is finite, i.e., $\#s(\mathcal{G}^{x})<\infty$. It suffices to show that every function $f\in H^{\infty}(\mathcal{G}^{x},P_{x})$ is constant. By Proposition \ref{RestrictionHarmonic} the restriction $f|_{\mathcal{G}_{x}^{x}}$ of such a function is harmonic with respect to the hitting measure $\nu$. By Lemma \ref{HittingMeasureIrreducible}, the measure $\nu$ is non-degenerate, so that our assumption implies $f|_{\mathcal{G}_{x}^{x}}\equiv C$ for some constant $C$. As in the proof of Proposition \ref{RestrictionHarmonic}, an application of Doob's optional stopping formula in Theorem \ref{DoobTheorem} gives 
\[
f(z)=\int_{\Omega}f\left(X_{T_{z}}(\omega)\right)d\mu_{z}(\omega)=\sum_{h\in\mathcal{G}_{x}^{x}}\mu_{z}(\{\omega\in\Omega\mid X_{T_{z}}(\omega)=h\})f(h)=C
\]
for every $z\in\mathcal{G}^{x}$. It follows that $f$ is constant and hence that $(G,\mu)$ is Choquet--Deny. 
\end{proof}

\subsection{Applications} \label{examples}

\subsubsection{The Choquet--Deny property and icc quotients}

Let $(\cG, \mu)$ and $(\cH, \nu)$ be discrete measured groupoids. We say that $\cH$ is a \emph{quotient} of $\cG$ if there exists a surjective Borel groupoid homomorphism $\rho: \cG \to \cH$ as in Proposition \ref{Prop: quotients under finite-to one maps remain CD}.

\begin{proposition}\label{Proposition:CD and icc quotients}
    A discrete measured groupoid is Choquet--Deny if and only if it admits no icc quotients apart from finite equivalence relations.
\end{proposition}
\begin{proof}
    Suppose first that $(\cG, \mu)$ is Choquet--Deny and let $(\cH, \nu)$ be an icc quotient. By Proposition \ref{Prop: quotients under finite-to one maps remain CD}, $(\cH, \nu)$ is also Choquet--Deny. By Corollary \ref{corollary: icc CD groupoids} it follows that $\cH$ must be a finite equivalence relation.

    Conversely, suppose that $(\cG, \mu)$ admits no icc quotients apart from finite equivalence relations. Since the equivalence relation associated with $\cG$ is a quotient of $\cG$, it must have finite orbits. Then we can proceed as in the proof of Corollary \ref{Corollary: CD passes to Iso}: using Proposition \ref{Proposition: semi-direct product groupoid}, we write $\cG$ as a crossed product of its isotropy groupoid by its equivalence relation. We then find a fundamental domain $Y \subseteq \cG\zero$ such that $\Iso(\cG)|_Y$ is a quotient of $\cG$. It follows that $\Iso(\cG)|_Y$ also has no icc quotients. But this means that for $\mu$-almost every $y \in Y$, the isotropy group $\cG_y^y$ has no icc quotients as a group. It follows that $\mu$-almost every $\cG_y^y$ is Choquet--Deny, which then also implies that for $\mu$-almost every $x \in \cG\zero$ the isotropy group $\cG_x^x$ is Choquet--Deny. By Corollary \ref{Corollary: bundle of groups CD} this means that $\Iso(\cG)$ is Choquet--Deny. But then Theorem \ref{Main Theorem CD} implies that $(\cG, \mu)$ is Choquet--Deny.
\end{proof}

\subsubsection{Transformation groupoids}

In this final subsection we classify the Choquet--Deny property for transformation groupoids. Corollary~\ref{Main Corollary} immediately follows from the following. 
\begin{theorem}\label{theorem: transformation groupoids} 
Let $\Gamma$ be a countable group acting on a Borel probability space $(X,\mu)$. Then the transformation groupoid $(\Gamma\ltimes X,\mu)$ is Choquet--Deny if and only if $\Gamma$ is Choquet--Deny and $\Gamma\curvearrowright X$ has finite orbits $\mu$-almost everywhere. 
\end{theorem} 
\begin{proof}
First assume that $(\Gamma\ltimes X,\mu)$ is Choquet--Deny. By Proposition~\ref{Prop: transformation groupoid CD =00003D> group CD}, it follows that $\Gamma$ is Choquet--Deny.

Conversely, suppose that $\Gamma$ is Choquet--Deny and that $\Gamma\curvearrowright X$ has finite orbits $\mu$-almost everywhere. Since $(\Gamma\ltimes X)_{x}^{x}$ is a subgroup of $\Gamma$ for all $x\in X$, it follows that $(\Gamma\ltimes X)_{x}^{x}$ is Choquet--Deny for $\mu$-almost every $x\in X$. The converse therefore follows again from Theorem \ref{Main Theorem CD}. 
\end{proof}
In the case of the group of integers, we obtain the following dichotomy.

\begin{corollary} Let $(X,\mu)$ be a Borel probability space and suppose that $\mathbb{Z}\curvearrowright(X,\mu)$ ergodically. Then exactly one of the following two statements holds:
\begin{enumerate}
\item The action $\mathbb{Z}\curvearrowright(X,\mu)$ is free. 
\item The transformation groupoid $(\mathbb{Z}\ltimes X,\mu)$ is Choquet--Deny. 
\end{enumerate}
\end{corollary} 
\begin{proof}
Since the action $\mathbb{Z}\curvearrowright(X,\mu)$ is ergodic, either the orbits are periodic $\mu$-almost everywhere or the action is free. In the periodic case, $(\mathbb{Z}\ltimes X,\mu)$ is Choquet--Deny due to Theorem~\ref{theorem: transformation groupoids}. 
\end{proof}

\emergencystretch3em
\printbibliography

\end{document}